\documentclass[12pt,a4paper]{article}
\usepackage{amsmath,amsfonts,amsthm,amssymb,mathrsfs,mathtools}
\usepackage{times}
\usepackage[colorlinks=true]{hyperref}
\usepackage[titletoc,title]{appendix}

\numberwithin{equation}{section}
\usepackage[usenames,dvipsnames]{xcolor}

\theoremstyle{plain}
  \newtheorem{thm}{Theorem}[section]
  \newtheorem{prop}[thm]{Proposition}
  \newtheorem*{prop*}{Proposition}
  \newtheorem{cor}[thm]{Corollary}
  \newtheorem{lem}[thm]{Lemma}
  \newtheorem{tbassu}{Assumption} 
\theoremstyle{remark}
  \newtheorem{rem}[thm]{Remark}
  \newtheorem*{rem*}{Remark}  
\theoremstyle{definition}
  \newtheorem{defn}[thm]{Definition}
  \newtheorem*{conv}{Convention of notation}

\newenvironment{assu}[1]
  {\renewcommand\thetbassu{#1}\tbassu}
  {\endtbassu}

\newcommand{\apa}{\alpha}
\newcommand{\ba}{\beta}
\newcommand{\ga}{\gamma}
\newcommand{\Ga}{\Gamma}
\newcommand{\da}{\delta}
\newcommand{\Da}{\Delta}
\newcommand{\vn}{\varepsilon}
\newcommand{\ta}{\theta}

\newcommand{\ka}{\kappa}
\newcommand{\la}{\lambda}
\newcommand{\La}{\Lambda}
\newcommand{\sa}{\sigma}
\newcommand{\vi}{\varphi}
\newcommand{\oa}{\omega}

\newcommand{\pl}{\partial}
\newcommand{\fc}{\frac}
\newcommand{\na}{\nabla}
\newcommand{\ol}{\overline}
\newcommand{\wt}{\widetilde}

\newcommand{\mb}{\mathbb}
\newcommand{\ml}{\mathcal}

\newcommand{\EQ}[1]{\begin{equation}\begin{split}#1\end{split}\end{equation}}
\newcommand{\R}{\mathbb{R}}

\begin{document}

\title{Mixed dimensional infinite soliton trains for nonlinear
Schr\"{o}dinger equations}

\author{
Liren Lin%
\thanks{Institute of Mathematics,
Academia Sinica,
Taipei, Taiwan, 
E-mail address: \texttt{b90201033@ntu.edu.tw}}
\and Tai-Peng Tsai%
\thanks{
Department of Mathematics, University of British Columbia, Vancouver BC
Canada V6T 1Z2, 
and Center for Advanced Study in Theoretical Sciences,
National Taiwan University,
Taipei, Taiwan, E-mail address: \texttt{ttsai@math.ubc.ca}}
}

\date{}

\maketitle

\begin{abstract}
In this note we construct mixed dimensional infinite soliton 
trains, which are solutions of nonlinear Schr\"odinger 
equations whose asymptotic profiles at time infinity consist of 
infinitely many solitons of multiple dimensions. For example
infinite line-point soliton trains in 2D space,
and infinite plane-line-point soliton trains in 3D space. This 
note extends the works of Le Coz, Li and Tsai 
\cite{LeCoz_Li_Tsai,LeCoz_Tsai}, where single dimensional trains
are considered. In our approach, spatial $L^\infty$ bounds for 
lower dimensional trains play an essential role.



\end{abstract}


\section{Introduction}

In this paper, we consider the nonlinear Schr\"{o}dinger equation
\begin{equation}
i\pl_{t}u+\Da u+f(u)=0,\label{eq:NLSE}
\end{equation}
where $u=u(t,x)$ is a complex-valued function on 
$\mb R\times\mb R^{d}$, $d\ge 1$, and  $f:\mb C \to \mb C$ is the 
nonlinearity. Our goal is to construct 
\emph{mixed dimensional infinite soliton trains} (mixed trains), 
which are solutions of \eqref{eq:NLSE} whose asymptotic profiles 
at time infinity consist of infinitely many solitons of multiple
dimensions.

The nonlinear Schr\"{o}dinger equation \eqref{eq:NLSE} appears in 
various physical contexts, for example in nonlinear optics or in 
the modelling of Bose-Einstein condensates. Mathematically speaking, 
it is one of the model nonlinear dispersive PDE, along with the 
Korteweg-De Vries equation and the nonlinear wave equation. 
Its local Cauchy theory in the energy space $H^1(\R^d)$ is well 
understood (see e.g. \cite{Ca03} and the references cited therein). 
Its long time dynamics has two competing effects: First of all, if 
the nonlinearity is not too strong, the linear part of the equation 
can dominate and solutions may behave as if they were solutions to 
the free linear Schr\"odinger equation. This is the 
\emph{scattering effect}. On the other hand, in some cases the 
nonlinear term dominates and the solution tends to concentrate, with
possible blow-up in finite time. This is the \emph{focusing effect}. 
At the equilibrium between these two effects, one may encounter many
different types of structures that neither scatter nor focus. The most
common of these non-scattering global structures are the solitons, but
there exist also dark solitons, kinks, etc. A generic conjecture for
nonlinear dispersive PDE is the \emph{Soliton Resolution Conjecture}. 
Roughly speaking, it says that, as can be observed in physical 
settings, any global solution will eventually decompose at large time 
into a scattering part and well separated non-scattering structures, 
usually a sum of solitons. Apart from integrable cases (see e.g. 
\cite{ZaSh72}), such conjecture is usually out of reach. Intermediate 
steps toward this conjecture are existence and stability results of 
configurations with well separated non-scattering structure, like 
multi-solitons, multi-kinks, infinite soliton and kink-soliton trains, 
etc. See \cite{LeCoz-Tsai-survey} for a survey on these subjects. 
In what follows we describe results most relevant to us.

Multi-solitons are solutions of \eqref{eq:NLSE} with the asymptotic 
profile 
\EQ{
\label{eq:multi-soliton}
T(t,x) = \sum _{j=1}^N R_j(t,x)
}
as $t \to \infty$, where $N\ge 2$ and each $R_j$ is a soliton to be 
specified in \eqref{eq:expofsol}. The first result of existence of 
multi-solitons was obtained in Zakharov and Shabat \cite{ZaSh72} in 
the case of the $1$-d focusing cubic (i.e. $d=1$, $f(z)=|z|^2z$) 
nonlinear Schr\"odinger equation via the inverse scattering method. 
Indeed, in this particular case the equation is completely 
integrable and one can obtain multi-solitons in a rather explicit 
manner. Kamvissis \cite{Ka95} showed that it is possible to push 
the inverse scattering analysis forward and obtain the existence of 
an \emph{infinite soliton train}, i.e. a solution $u$ of 
\eqref{eq:NLSE} defined as in \eqref{eq:multi-soliton} but with 
$N=+\infty$. In fact, it is shown that, under some technical 
hypotheses, any solution to \eqref{eq:NLSE} with initial data in the 
Schwartz class will eventually decompose at large time as an infinite 
soliton train and a ``background radiation component''. There are 
also results for multi-dark solitons for the companion integrable 
Gross-Pitaevskii case, i.e. $d=1$ and $f(z)=(1-|z|^2)z$, but no known 
results for infinite trains.

In a  non-integrable setting,
the first existence result of multi-solitons was obtained by Merle in 
\cite{Me90} as a by-product of the proof of existence of multiple 
blow-up points solutions for $L^2$-critical \eqref{eq:NLSE}, i.e. 
$f(z) = |z|^{4/d} z$. The techniques initiated in  \cite{Me90} were 
then developed in \cite{CoLC11,CoMaMe11,MaMe06,MaMeTs06} for other 
nonlinearities. 
The idea, so called the energy method, is to choose an increasing 
sequence of time $(t_n)$ with $t_n \to+\infty$ and consider the 
solutions $(u_n)$ to \eqref{eq:NLSE} which solve the equation 
backward in time with final data $u_n(t_n) = T(t_n)$. The sequence 
$(u_n)$ is an approximate sequence for a multi-soliton. To show its
convergence, two arguments are at play. First, one shows that there 
exists a time $t_0$ independent of $n$ such that $u_n$ satisfies on 
$[t_0, t_n]$ the uniform estimates
\[
\|(u_n-T)(t)\|_{H^1}\le e^{-\mu \sqrt {\omega_*} v_* t}.
\]
Second, we have compactness of the sequence of initial data $u_n(t_0)$, i.e. there exists $u_0$ so that 
$u_{n}(t_0) \to u_0$ in $H^s$ for all $0<s<1$. 
See also \cite{Pe04,RoScSo03,MaMeTs06} for stability results 
under restrictive hypotheses.

The energy method is very flexible and can be adapted to other situations. However, its implementation is far from being trivial when the number of solitons is infinite or when one soliton is replaced by a kink. In Le Coz, Li and Tsai \cite{LeCoz_Li_Tsai,LeCoz_Tsai}, an approach based on fixed point argument has been used to construct such structures. In this approach, the large relative speed has been used to get smallness of the Duhamel term due to short interaction time. It is however delicate when the gradient of the error term is also measured. We will explain this approach in more details below, as we will use it to construct mixed dimensional infinite soliton trains.

We now make two assumptions on the nonlinearity $f$, 
which will be assumed throughout the paper.

\begin{assu}{(F)}\label{assuF}
$f(z)=g(|z|^{2})z$, where 
$g\in C([0,\infty),\mb R)\cap C^{2}((0,\infty),\mb R)$
satisfies $g(0)=0$, and
\begin{equation}
|sg'(s)|+|s^{2}g''(s)|\le C_{0}(s^{\apa_{1}/2}+s^{\apa_{2}/2})\quad(s>0)\label{eq:C0}
\end{equation}
for some $C_{0}>0$, and some $\apa_{1},\apa_{2}$ satisfying
\[
0<\apa_{1}\le\apa_{2}<\apa_{\max}=\left\{ \begin{array}{cl}
\infty & \mbox{if }d=1,2\\
\fc{4}{d-2} & \mbox{if }d\ge3.
\end{array}\right.
\]
\end{assu}

If a nontrivial $\vi\in H^{1}(\mb R^{d},\mb R)$ (bound state) and
an $\oa>0$ (frequency) satisfy 
\begin{equation}
-\Da\vi+\oa\vi=f(\vi),\label{eq:seq}
\end{equation}
then for any $v\in\mb R^{d}$ (velocity), $x^{0}\in\mb R^{d}$ (initial
position), and $\ga\in\mb R$ (phase), 
\begin{equation}
R_{\vi,\oa,v,x^{0},\ga}(t,x)\coloneqq 
e^{i(\oa t+\fc{1}{2}v\cdot x-\fc{1}{4}|v|^{2}+\ga)}\vi(x-x^{0}-vt)\label{eq:expofsol}
\end{equation}
is a solution of \eqref{eq:NLSE}, called \emph{soliton} in this paper,
 in its broader meaning of \emph{solitary wave}.
The existence of solitons is a property of the nonlinearity $f$.
To construct infinite soliton trains, we assume that there is a one parameter family of arbitrarily ``small'' solitons:

\begin{assu}{(T)$_d$}\label{assuT}
For given dimension $d$, there are $\oa_{*}>0$, $0<a<1$, 
and $D>0$ (each depends only on $d,f$) such that for 
$0<\oa<\oa_{*}$, there exist nontrivial solutions 
$\vi=\vi_\omega\in H^{1}(\mb R^{d},\mb R)$ of 
\eqref{eq:seq} satisfying
\begin{equation}\label{eq:deofphi}
|\vi(x)|+\oa^{-\fc{1}{2}}|\na\vi(x)|
\le D\oa^{\fc{1}{\apa_{1}}}e^{-a\oa^{1/2}|x|},
\quad\forall x\in\mb R^{d}.
\end{equation}
\end{assu}

Assumption \ref{assuT} is true for a large set of nonlinearities. 
A typical example is 
\EQ{
g(s)=s^{\apa_{1}/2}+c s^{\apa_{2}/2},
}
where $c\in\mb R$, and $0<\apa_{1}<\apa_{2}<\apa_{\max}$;
see \cite[Proposition 2.1]{LeCoz_Tsai} for more general nonlinearities. We shall however take it as an assumption.

In the following we discuss our problem and approach in more details. 

\subsection{General idea}

Suppose $\{W_{j}(t,x)\}$ is a (finite or infinite) collection of
solutions of \eqref{eq:NLSE}. Intuitively, if these solutions are
sufficiently separated from each other, then the nonlinear effects
of their interactions should be negligible, and $\sum_{j}W_{j}$ should
be close to a solution. We are interested in the possibility that
$\sum_{j}W_{j}+\eta$ is a solution for some error $\eta=\eta(t,x)$
tending to zero as $t\to\infty$. The equation of $\eta$ is hence
\[
\left\{ \begin{aligned} & 
\textstyle i\pl_{t}\eta+\Da\eta+f(\sum_{j}W_{j}+\eta)-\sum_{j}f(W_{j})=0\\
 & \eta|_{t=\infty}=0\quad(\mbox{formally}).
\end{aligned}
\right.
\]
By Duhamel's principle, it suffices to solve the fixed point 
problem 
\begin{equation}
\eta(t)=\Phi\eta(t)
\coloneqq-i\int_{t}^{\infty}e^{i(t-\tau)\Da}[G(\tau)+H(\tau)]\, d\tau,
\label{eq:Duhamel}
\end{equation}
where 
\begin{align*}
G &\textstyle =f(\sum_{j}W_{j}+\eta)-f(\sum_{j}W_{j}),\\
H &\textstyle =f(\sum_{j}W_{j})-\sum_{j}f(W_{j}).
\end{align*}
For a specific profile $\sum_{j}W_{j}$ in our construction, 
we will try to prove that $\Phi$ is a contraction mapping on 
a closed ball of some Banach space (see \eqref{eq:norm1} and \eqref{eq:norm2} below for examples), with the needed 
inequalities derived from the standard dispersive 
estimates and the Strichartz estimates. In doing so, the 
above decomposition of the source term into $G$ and $H$ will 
be convenient. Apparently, the control of $G$ will come from 
our assumed control on $\eta$. On the other hand, the 
control of $H$ is much more elaborate. It will rely on our 
assumption that different $W_j$'s are sufficiently separated 
from each other. See the next section.

\subsection{Infinite soliton trains \label{sub:IST}}

By an \emph{infinite soliton train} we mean a solution of 
\eqref{eq:NLSE} whose asymptotic profile is of the form 
$T=\sum_{j\in\mb N}R_{j}$, where  
$R_{j}=R_{\vi_{j},\oa_{j},v_{j},x_{j}^{0},\ga_{j}}(t,x)$ 
are solitons as given by \eqref{eq:expofsol}.
We remark that the term
``train'' is only used in a suggestive sense. It well describes
the one dimensional situation where all solitons travel in the same
direction. In higher dimensions, the traveling
directions of the constituting solitons can be rather arbitrary. 

Consider $\{W_{j}\}$ in \eqref{eq:Duhamel} to be such 
$\{R_{j}\}_{j\in\mb N}$, 
with $x^0_j=0$ for all $j$ 
(see Remark \ref{rem:ini} below).
To give an idea of the kind of things to be proved, 
we cite two results (rephrased).
\begin{itemize}
\item \cite[Theorem 1.2]{LeCoz_Tsai} For $\la>0$,
let $X=X_{\la}$ be the Banach space of all 
$\eta:[0,\infty)\times\mb R^{d}\to\mb C$
satisfying 
\begin{equation}
\|\eta\|_{X}\coloneqq\sup_{t\ge0}e^{\la t}
(\|\eta(t)\|_{L_{x}^{2+\apa_{2}}}+\|\eta\|_{S(t)})<\infty.\label{eq:norm1}
\end{equation}
Suppose $\fc{\apa_{2}}{2+\apa_{2}}\le\apa_{1}$. Then,
for $\la$ large enough and under suitable conditions of 
$\{\oa_j\}$ and $\{v_j\}$, $\Phi$ is a contraction
mapping on the closed unit ball of $X_{\la}$.
\item \cite[Theorem 1.6]{LeCoz_Tsai} Fix any $t_0>0$. 
For $\la,c>0$, let $X=X_{\la,c}$ be the Banach space of all 
$\eta:[t_0,\infty)\times\mb R^{d}\to\mb C$
satisfying 
\begin{equation}
\|\eta\|_{X}\coloneqq\sup_{t\ge t_0}(e^{\la t}\|\eta\|_{S(t)}
+e^{c\la t}\|\na\eta\|_{S(t)})<\infty.\label{eq:norm2}
\end{equation}
Suppose $0<\apa_{1}<\fc{4}{d+2}$. Then, for suitable $c\le 1$
and large enough $\la$,
and under suitable conditions
of $\{\oa_j\}$ and $\{v_j\}$, $\Phi$ is a contraction mapping on 
the closed unit ball of $X_{\la,c}$.
\end{itemize}
The $S(t)$ in the statements represents the Strichartz space on the
time interval $[t,\infty)$ (relevant preliminaries will be given).
What the ``suitable conditions'' are will be clear in due course.
Roughly speaking, they concern the speeds of the following two limits:
i) $\oa_{j}\to 0$, so that the Lebesgue norms of $T$ (or also of $\na T$) 
can be controlled by \eqref{eq:deofphi};
ii) $|v_{j}-v_{k}|\to\infty$ (as $j,k\to\infty$, $j\ne k$), 
so that different solitons are sufficiently separated to each other.

\begin{rem}\label{rem:ini}
We assume the initial positions of all the solitons 
to be the origin for simplicity. This is an apparent 
reason why some constructions
(such as the second result cited above)
have to be done on time intervals $[t_0,\infty)$ 
with positive $t_0$. The same situation will occur in
some of our results for mixed dimensional trains.
\end{rem}

\subsection{Mixed dimensional trains \label{sub:mixed}}

Now we consider asymptotic profiles consisting of solitons of 
multiple dimensions. The simplest example is 
\begin{equation}
T_{1}+T_{2}\coloneqq\sum_{k\in\mb N}R_{1;k}(t,x_{1})+\sum_{j\in\mb N}R_{2;j}(t,x_{1},x_{2}),\label{eq:1-2Train}
\end{equation}
where $R_{1;k}$ and $R_{2;j}$ are solitons in $\mb R_{x}^{1}$ and
$\mb R_{x}^{2}$ respectively. For convenience, we'll call them 1D
solitons and 2D solitons. \eqref{eq:1-2Train} can be visualized as
the profile of a line-point soliton train in $\mb R_{x}^{2}$ (and
a plane-line soliton train in $\mb R_{x}^{3}$, a space-plane soliton
train in $\mb R_{x}^{4}$, and so on). Similarly, we can consider
a combination of $e$D solitons and $d$D solitons for $1\le e<d$,
or even combinations involving three or more dimensions. (It turns
out that there are limited realizable combinations.
See the section ``Main results'' below.) Solutions having
such kind of profiles will be called \emph{mixed (dimensional) trains}.
In the following we take \eqref{eq:1-2Train} as an example to describe
the particular difficulties in constructing mixed trains.

First, our general idea encounters a problem if we only use a 2D error
$\eta(t,x_{1},x_{2})$. To see this, note that by posing a solution
of the form $T_{1}+T_{2}+\eta$, we get 
\[
H=f(T_{1}+T_{2})-\sum_{k}f(R_{1;k})-\sum_{j}f(R_{2;j}).
\]
Then, with $x_{1}$ being fixed, we have 
\[
\lim_{x_{2}\to\infty}H=f(T_{1}(t,x_{1}))-\sum_{k}f(R_{1;k}(t,x_{1})),
\]
which is nonzero in general. That is $H$ has no space decay at infinity,
and hence defies any suitable estimate (we will need $L^p_x$ controls 
of $H$ for $p\le 2$). To resolve this problem, we
will also introduce a lower dimensional error. Precisely, we will
construct a solution of the form 
\begin{align*}
T_{1}+\eta_{1}+T_{2}+\eta,
\end{align*}
where
$\eta_{1}=\eta_{1}(t,x_{1})$ is such that $T_{1}+\eta_{1}$ is an
1D train (i.e. a solution of \eqref{eq:NLSE}). In this way, by regarding
$\{W_{j}\}$ as the sequence defined by $W_{1}=T_{1}+\eta_{1}$, and
$W_{j+1}=R_{2;j}$ for $j\in\mb N$, we have 
\[
H=f(T_{1}+\eta_{1}+T_{2})-f(T_{1}+\eta_{1})-\sum_{j}f(R_{2;j}),
\]
which we will be able to estimate suitably.

The main difficulty in the construction is that the 1D objects $R_{1;k}$
and $\eta_{1}$ only allow $L^{\infty}$ bounds in $x_{2}$. There
are two aspects of the effect of this restriction.

1) To estimate products involving 1D objects and $\eta$ (such as
$\||\eta||\eta_{1}|^{\apa_{i}}\|_{L^p_x}$), we must have 
$L_{x_{1}}^{\infty}$
estimates of the 1D objects, to avoid the need of dealing with ``anisotropic''
estimates of $\eta$. Here by an \emph{anisotropic estimate} we mean an 
$L_{x_{1}}^{p_{1}}L_{x_{2}}^{p_{2}}$
estimate with $p_{1}\ne p_{2}$. Whether such estimates are available
for $\eta$ is unclear (see Appendix \ref{appII}). Now, for $R_{1;k}$, the
$L_{x_{1}}^{\infty}$ estimate is easy to obtain from \eqref{eq:deofphi}.
However, there is no ready result asserting an $L_{x_{1}}^{\infty}$ control 
of $\eta_{1}$. In the previous works \cite{LeCoz_Li_Tsai,LeCoz_Tsai},
the authors did not concern the possibility of constructing (single
dimensional) trains with $L_{x}^{\infty}$ control of the errors.
(Nevertheless, \eqref{eq:norm2} does imply such controls by Sobolev
embedding. We'll discuss it in Section \ref{sub:PLP}.) As a 
consequence, we will investigate this problem before going into the 
mixed cases. 

2) On the other hand, anisotropic estimates for $R_{2;j}$ 
(and $\na R_{2;j}$) are easy to obtain 
(also by \eqref{eq:deofphi}). In estimating products
of them with the 1D objects, we will exploit such estimates.
As will be seen, using anisotropic estimates does give 
us much better results.

\subsection{Main results}

We summarize our main results in the following. 

\begin{enumerate}
\item 
From Theorem \ref{thm:single1} and Theorem \ref{thm:single2} (see
also Corollary \ref{cor:single-sum}), there exist single dimensional
trains $T+\eta$ such that
 \begin{itemize}
 \item $\|\eta(t)\|_{L_{x}^{2}\cap L_{x}^{\infty}}$ has exponential decay
 in $t$, provided
  \begin{itemize}
  \item $d=1$, with $0<\apa_{1}\le\apa_{2}<\apa_{\max}$;
  \item $d=2,3$, with $2(\fc{1}{2}-\fc{1}{d})<\apa_{1}<2$ 
  and $\apa_{1}\le\apa_{2}<\apa_{\max}$.
  \end{itemize} 
 \item $\|\eta(t)\|_{H_{x}^{1}\cap W_{x}^{1,\infty}}$ has exponential decay
 in $t$, provided
  \begin{itemize}
  \item $d=1$, with $1\le\apa_{1}<2$ and $\apa_{1}\le\apa_{2}<\apa_{\max}$.
  \end{itemize}
 \end{itemize}
\item
With the above existence results of $e$D trains $T_{e}+\eta_{e}$
($e$ corresponds to the above $d$), Theorem \ref{thm:mixed0} and
Theorem \ref{thm:mixed1} assert the existence of $e$D-$d$D trains
$T_{e}+\eta_{e}+T_{d}+\eta$ such that
 \begin{itemize}
 \item $\|\eta\|_{S(t)}$ has exponential decay in $t$, provided
  \begin{itemize}
  \item $1\le e\le 3$, $e<d\le e+3$, with 
  $\max(2(\fc{1}{2}-\fc{1}{e}),0)<\apa_1\le \apa_2 \le 4/d$.
  \end{itemize}
 \item $\|\eta\|_{S(t)}+\|\na\eta\|_{S(t)}$ has exponential decay in $t$,
 provided
  \begin{itemize}
  \item $e=1$, $d=2$, with $1\le\apa_{1}<4/3$ and 
  $\apa_{1}\le\apa_{2}<\infty$.
  \end{itemize}
 \end{itemize}
\item
With the last result, Theorem \ref{thm:1D2D3D} asserts the
existence of 1D-2D-3D trains 
$T_{1}+\eta_{1}+T_{2}+\eta_{2}+T_{3}+\eta$
such that 
 \begin{itemize}
 \item $\|\eta\|_{S(t)}$ has exponential decay in $t$, provided 
 $1\le\apa_{1}<4/3$ and $\apa_{1}\le\apa_{2}\le 4/3$.
 \end{itemize}
\end{enumerate}

We give some remarks on other possible constructions, not 
treated in this paper.

\begin{rem}
We focus on infinite soliton trains in this paper. Our method
can apparently be used to construct trains with finitely many
solitons. In that case, there is no need of any assumptions 
on (the finite sequences) $\{\oa_j\}$ and $\{v_j\}$, as long 
as \eqref{eq:deofphi} is valid for all the solitons.
\end{rem}

\begin{rem}
We may add a half kink $K(t,x_1)$  (if it exists) to one 
side of $T_1$ as in \cite{LeCoz_Li_Tsai, LeCoz_Tsai},
if all solitons (1D and higher dimensional) 
are positioned in the other side.
If $T_1$ is finite, 
we may add half kinks to both sides. (See 
\cite[Figure 1]{LeCoz_Li_Tsai} for an illustration.)
Notice that, in this case, it is still possible that there are
infinitely many higher dimensional solitons. 
For example, consider the infinite 2D train profile
$T_2$, with $v_j=(v_{j1},v_{j2})$ being the velocities 
of the solitons. To combine it with $T_1$ having kinks on both
sides, we can arrange $v_{j1}$ to 
make $T_2$ well separated from $T_1$, 
and take $|v_{j2}- v_{\ell 2}|\to \infty$ as 
$j,\ell\to \infty$ ($j\ne\ell$) to make the 2D solitons to be
separated from each other.
\end{rem}

The rest of the paper is organized as follows: 
In Section \ref{sec:BI},
we collect some basic inequalities of the nonlinearity $f$. 
In Section \ref{sec:single}, we construct single dimensional 
trains with spatial supremum control on the errors.
Along the way, we give some detailed discussions
as to the control of trains, which are also fundamental for mixed
dimensional cases. 
We begin Section \ref{sec:MixedT} by showing the importance
of using the Strichartz estimates for constructing mixed trains. 
Section \ref{sub:STR} gives the necessary preliminaries related 
to the Strichartz space. The $e$D-$d$D trains are considered in 
Section \ref{sub:eDdDtrain}, and finally the 1D-2D-3D trains are 
constructed in Section \ref{sub:PLP}.

\section{Basic inequalities \label{sec:BI}}

In this section, we collect some inequalities that are simple 
consequences of Assumption \ref{assuF}. The only thing 
that can be said new is Proposition \ref{prop:decomp} 
(and Corollary \ref{cor:decomp}), of which the flexibility in 
choosing the powers will be useful in some places. We first 
make the following 

\begin{conv}

In this paper, a constant is called \emph{universal} if it depends
only on the dimension $d$ and the nonlinearity $f$, in particular
$C_{0},\apa_{1},\apa_{2}$ in Assumption \ref{assuF} and 
$\oa_{*},a,D$ in Assumption \hyperref[assuT]{(T)$_d$}. We will use the 
notation $\lesssim$
in the sense that the inequality is up to a 
\emph{universal multiplicative constant}. 
The dependence on other parameters will be given explicitly,
possibly as a subscript of $\lesssim$.

\end{conv}

Let $f_{z}\coloneqq\fc{1}{2}(\fc{\pl f}{\pl x}-i\fc{\pl f}{\pl y})$
and $f_{\bar{z}}\coloneqq\fc{1}{2}(\fc{\pl f}{\pl x}+i\fc{\pl f}{\pl y})$,
where $f$ is regarded as a function of $(x,y)\in\mb R^{2}$ by letting 
$f(x,y)\coloneqq f(x+iy)$. 
\begin{prop}
For $w_{1},w_{2}\in\mb C$, we have 
\begin{equation}\label{eq:fineq0}
|f(w_{1}+w_{2})-f(w_{1})|\lesssim
\sum_{i=1,2}(|w_{2}||w_{1}|^{\apa_{i}}+|w_{2}|^{\apa_{i}+1}),
\end{equation}
and
\begin{equation}
\begin{aligned} & |f_{z}(w_{1}+w_{2})-f_{z}(w_{1})|+|f_{\bar{z}}(w_{1}+w_{2})-f_{\bar{z}}(w_{1})|\\
 & \qquad\qquad\qquad\lesssim\sum_{i=1,2}|w_{2}|^{\min(\apa_{i},1)}(|w_{1}|+|w_{2}|)^{\max(\apa_{i}-1,0)}.
\end{aligned}
\label{eq:fineq1}
\end{equation}

\end{prop}
See \cite[Lemma 2.2]{LeCoz_Li_Tsai} for the proofs of both inequalities.
Notice that since $f(0)=0$, \eqref{eq:fineq0} subsumes $|f(w)|\lesssim\sum_{i=1,2}|w|^{\apa_{i}+1}$
for $w\in\mb C$. It's easy to check that we also have $f_{z}(0)=f_{\bar{z}}(0)=0$
from Assumption \ref{assuF}, and \eqref{eq:fineq1} subsumes
$|f_{z}(w)|+|f_{\bar{z}}(w)|\lesssim\sum_{i=1,2}|w|^{\apa_{i}}$ for
$w\in\mb C$.

For $w:\mb R^{d}\to\mb C$ such that the chain rule applies to $\na f(w(x))$
(e.g. $w\in W_{loc}^{1,1}$), it's easy to check that 
\begin{equation}
\na(f(w(x)))=f_{z}(w(x))\na w(x)+f_{\bar{z}}(w(x))\ol{\na w(x)}.\label{eq:fchain}
\end{equation}
We have the following corollary.
\begin{prop}
For $w_{1},w_{2}\in W_{loc}^{1,1}(\mb R^{d},\mb C)$, we have
\begin{align}
 & |\na[f(w_{1}+w_{2})-f(w_{1})]|\nonumber \\
 & \quad\lesssim\sum_{i=1,2}\left\{ |w_{2}|^{\min(\apa_{i},1)}(|w_{1}|+|w_{2}|)^{\max(\apa_{i}-1,0)}|\na w_{1}|+(|w_{1}|+|w_{2}|)^{\apa_{i}}|\na w_{2}|\right\} ,\label{eq:fineq1'}
\end{align}
and 
\begin{align}
 & |\na[f(w_{1}+w_{2})-f(w_{1})-f(w_{2})]|\nonumber \\
 & \quad\lesssim\sum_{i=1,2}(|w_{1}|+|w_{2}|)^{\max(\apa_{i}-1,0)}(|w_{2}|^{\min(\apa_{i},1)}|\na w_{1}|+|w_{1}|^{\min(\apa_{i},1)}|\na w_{2}|).\label{eq:fineq1''}
\end{align}
 \end{prop}
\begin{proof}
Let $w=w_{1}+w_{2}$. By \eqref{eq:fchain}, 
\begin{align*}
|\na[f(w)-f(w_{1})]| & =|f_{z}(w)\na w+f_{\bar{z}}(w)\ol{\na w}-f_{z}(w_{1})\na w_{1}-f_{\bar{z}}(w_{1})\ol{\na w_{1}}|\\
 & \le(|f_{z}(w)-f_{z}(w_{1})|+|f_{\bar{z}}(w)-f_{\bar{z}}(w_{1})|)|\na w_{1}|\\
 & \qquad\qquad+(|f_{z}(w)|+|f_{\bar{z}}(w)|)|\na w_{2}|,
\end{align*}
and \eqref{eq:fineq1'} follows \eqref{eq:fineq1}. For \eqref{eq:fineq1''},
we have
\begin{align*}
|\na[f(w)-f(w_{1})-f(w_{2})]| & =|f_{z}(w)\na w+f_{\bar{z}}(w)\ol{\na w}\\
 & \qquad\qquad-\sum_{j=1,2}(f_{z}(w_{j})\na w_{j}+f_{\bar{z}}(w_{j})\ol{\na w_{j}})|\\
 & \le\sum_{j=1,2}(|f_{z}(w)-f_{z}(w_{j})|+|f_{\bar{z}}(w)-f_{\bar{z}}(w_{j})|)|\na w_{j}|.
\end{align*}
Let $1'=2$ and $2'=1$. Then \eqref{eq:fineq1} implies
\begin{align*}
 & |\na[f(w)-f(w_{1})-f(w_{2})]|\\
 & \quad\lesssim\sum_{j=1,2}\Bigl\{\sum_{i=1,2}|w_{j'}|^{\min(\apa_{i},1)}(|w_{1}|+|w_{2}|)^{\max(\apa_{i}-1,0)}\Bigr\}|\na w_{j}|\\
 & \quad=\sum_{i=1,2}(|w_{1}|+|w_{2}|)^{\max(\apa_{i}-1,0)}(|w_{2}|^{\min(\apa_{i},1)}|\na w_{1}|+|w_{1}|^{\min(\apa_{i},1)}|\na w_{2}|).
 \qedhere
\end{align*}
\end{proof}

\begin{prop}\label{prop:decomp} 
For any $\ta_{ij},\phi_{ij}\in[0,1]$ ($i=1,2$, $j\in\mb N$), and 
for any absolutely convergent series $\sum_{j\in\mb N}w_{j}$ of complex
numbers, we have
\[
|f(\sum_{j}w_{j})-\sum_{j}f(w_{j})|\lesssim\sum_{i=1,2}\sum_{j}\Bigl(
|w_{j}|^{\apa_{i}+\ta_{ij}}(\sum_{\ell\ne j}|w_{\ell}|)^{1-\ta_{ij}}
+|w_{j}|^{1-\phi_{ij}}(\sum_{\ell\ne j}|w_{\ell}|)^{\apa_{i}+\phi_{ij}}\Bigr).
\]
\end{prop}
\begin{rem}
It should be clear that we use $\sum_{j}$ to represent $\sum_{j\in\mb N}$,
and $\sum_{\ell\ne j}$ (with $j$ fixed) to represent 
$\sum_{\ell\in\mb N\setminus\{j\}}$.
We'll freely use such simplified notation in this paper.\end{rem}
\begin{proof}
The inequality is trivial if $w_{j}=0$ for all $j$. So assume at
least one $w_{j}\ne0$. Let $h_{j}=|w_{j}|/(\sum_{\ell}|w_{\ell}|)$
for each $j\in\mb N$, and let $w=\sum_{j}w_{j}$. We have
\begin{align*}
|f(w)-\sum_{j}f(w_{j})| & =|\sum_{j}[h_{j}f(w)-f(w_{j})]|\\
 & \le\sum_{j}\Bigl\{ h_{j}|f(w)-f(w_{j})|+(1-h_{j})|f(w_{j})|\Bigr\}.
\end{align*}
By \eqref{eq:fineq0}, 
\begin{align*}
h_{j}|f(w)-f(w_{j})| & \lesssim\fc{|w_{j}|}{\sum_{\ell}|w_{\ell}|}\sum_{i=1,2}\left\{ |w-w_{j}||w_{j}|^{\apa_{i}}+|w-w_{j}|^{\apa_{i}+1}\right\} \\
 & \le\fc{|w_{j}|(\sum_{\ell\ne j}|w_{\ell}|)}{\sum_{\ell}|w_{\ell}|}\sum_{i=1,2}\Bigl\{|w_{j}|^{\apa_{i}}+(\sum_{\ell\ne j}|w_{\ell}|)^{\apa_{i}}\Bigr\}.
\end{align*}
And
\[
(1-h_{j})|f(w_{j})|\lesssim\fc{\sum_{\ell\ne j}|w_{\ell}|}{\sum_{\ell}|w_{\ell}|}\sum_{i=1,2}|w_{j}|^{\apa_{i}+1}=\fc{|w_{j}|(\sum_{\ell\ne j}|w_{\ell}|)}{\sum_{\ell}|w_{\ell}|}\sum_{i=1,2}|w_{j}|^{\apa_{i}}.
\]
Thus 
\[
|f(w)-\sum_{j}f(w_{j})|\lesssim\sum_{i=1,2}\sum_{j}\fc{|w_{j}|(\sum_{\ell\ne j}|w_{\ell}|)}{\sum_{\ell}|w_{\ell}|}\Bigl\{|w_{j}|^{\apa_{i}}+(\sum_{\ell\ne j}|w_{\ell}|)^{\apa_{i}}\Bigr\}.
\]
Now fix any $\ta\in [0,1]$. Notice that by Young's inequality we have
\[
x+y\ge(1-\ta)^{-(1-\ta)}\ta^{-\ta}x^{1-\ta}y^{\ta}\ge x^{1-\ta}y^{\ta},
\quad\forall x,y\ge 0.
\]
Thus
\[
\fc{|w_{j}|(\sum_{\ell\ne j}|w_{\ell}|)}{\sum_{\ell}|w_{\ell}|}=\fc{|w_{j}|(\sum_{\ell\ne j}|w_{\ell}|)}{|w_{j}|+(\sum_{\ell\ne j}|w_{\ell}|)}\le\fc{|w_{j}|(\sum_{\ell\ne j}|w_{\ell}|)}{|w_{j}|^{1-\ta}(\sum_{\ell\ne j}|w_{\ell}|)^{\ta}}=|w_{j}|^{\ta}(\sum_{\ell\ne j}|w_{\ell}|)^{1-\ta}.
\]
This completes the proof.\end{proof}

\begin{cor}
\label{cor:decomp}For any $\ta_{i},\phi_{i}\in[0,1]$ ($i=1,2$), and 
$w_{1},w_{2}\in\mb C$, 
\[
|f(w_{1}+w_{2})-f(w_{1})-f(w_{2})|\lesssim\sum_{i=1,2}\Bigl(
|w_{1}|^{\apa_{i}+\ta_{i}}|w_{2}|^{1-\ta_{i}}
+|w_{1}|^{1-\phi_{i}}|w_{2}|^{\apa_{i}+\phi_{i}}\Bigr).
\]

\end{cor}
\begin{proof} 
The assertion follows by considering $w_j=0$ for $j\ge 3$, and taking
\[
\theta_{i2}  = \phi_{i1}  =\phi_i, \quad
\phi_{i2} = \theta_{i1}   =\theta_i
\]
in Proposition \ref{prop:decomp}.
\end{proof}

\section{Single dimensional trains with $L_{x}^{\infty}$ 
control of errors\label{sec:single} }

In this section we investigate the possibility of constructing single
dimensional trains 
\begin{align}\label{eq:T-def}
T+\eta
\end{align}
such that $\|\eta(t)\|_{L_{x}^{\infty}}$
(or even $\|\eta(t)\|_{W_{x}^{1,\infty}}$) decays exponentially in
$t$. Here
\begin{align*}
 T=\sum_{j\in\mb N}R_{j}, \quad \text{where}\quad
 R_{j}=R_{\phi_{j},\oa_{j},v_{j},x_{j}^{0}=0,\ga_{j}}(t,x)
\end{align*}
are $d$D solitons as given by \eqref{eq:expofsol}, with
$x_{j}^{0}=0$ for all $j$.
Besides the main results (Theorem \ref{thm:single1} and Theorem
\ref{thm:single2}), many discussions in this section are also useful
for next section.

By Assumption 
\ref{assuT},
\begin{equation}
\begin{aligned}|R_{j}(t,x)| & \le D\oa_{j}^{\fc{1}{\apa_{1}}}e^{-a\oa_{j}^{1/2}|x-v_{j}t|},\\
|\na R_{j}(t,x)| & \lesssim D\langle v_{j}\rangle
\oa_{j}^{\fc{1}{\apa_{1}}}e^{-a\oa_{j}^{1/2}|x-v_{j}t|},
\end{aligned}
\label{eq:raw}
\end{equation}
where we used
\[
|v_{j}|/2+\oa_{j}^{1/2}\lesssim 
\langle v_{j}\rangle, \quad \langle v\rangle:=(|v|^2+1)^{1/2}.
\]
By the change of variable $x=\oa^{-1/2}y$, we get for $0<p\le\infty$
\begin{equation}
\begin{aligned}\|R_{j}\|_{L_{x}^{p}} & \le D_{p}\oa_{j}^{\fc 1{\apa_{1}}-\fc d{2p}},\\
\|\na R_{j}\|_{L_{x}^{p}} & \lesssim D_{p}\langle v_{j}\rangle\oa_{j}^{\fc 1{\apa_{1}}-\fc d{2p}},
\end{aligned}
\label{eq:lprgr}
\end{equation}
where $D_{p}=D\|e^{-a|y|}\|_{L_{y}^{p}}$. 
\begin{rem}
The norm $\|\cdot\|_{L_{x}^{p}}$ in \eqref{eq:lprgr} is indeed $\|\cdot\|_{L^{\infty}(\mb R,L^{p}(\mb R^{d}))}$
($\|\cdot\|_{L_{t}^{\infty}L_{x}^{p}}$ for short). We shall however
maintain the sloppy notation for simplicity. The same remark applies
to $\|T\|_{L_{x}^{p}}$ and $\|\na T\|_{L_{x}^{p}}$, which will be
considered soon. Note that as solitons do not change shapes, they
can not have $L_{t}^{s}L_{x}^{p}$ bounds for any $s<\infty$.
\end{rem}

\begin{rem}
\label{rem:p_lb} Using the inequality 
$|y|\ge(|y_{1}|+\cdots+|y_{d}|)/\sqrt{d}$,
we get $D_{p}\le D(\fc{2\sqrt{d}}{ap})^{d/p}$. 
Thus, for fixed $p_0>0$,
$p\ge p_0$ implies $D_p\lesssim_{p_0} 1$. 
In particular, $D_p\lesssim 1$ if $p_0$ is universal.
There will be times we have to consider $p_0<1$.
\end{rem}
\begin{lem}
\label{lem:moving} For $0<p\le\infty$, and $M\ge\max(1,p^{-1})$,
we have 
\begin{align*}
\|\sum_{j}|R_{j}|\|_{L_{x}^{p}} & \le D_p\Bigl(\sum_{j}\oa_{j}^{\fc 1M(\fc 1{\apa_{1}}-\fc d{2p})}\Bigr)^{M},\\
\|\sum_{j}|\na R_{j}|\|_{L_{x}^{p}} & \lesssim D_p \Bigl(\sum_{j}\langle v_{j}\rangle^{\fc 1M}\oa_{j}^{\fc 1M(\fc 1{\apa_{1}}-\fc d{2p})}\Bigr)^{M}.
\end{align*}
\end{lem}
\begin{proof}
The first inequality is true by the following computation: 
\begin{align*}
\|\sum_{j}|R_{j}|\|_{L_{x}^{p}} & =\|(\sum_{j}|R_{j}|)^{1/M}\|_{L_{x}^{Mp}}^{M}\\
 & \le\|\sum_{j}|R_{j}|^{1/M}\|_{L_{x}^{Mp}}^{M}\quad(\mbox{since }1/M\le1)\\
 & \le(\sum_{j}\||R_{j}|^{1/M}\|_{L_{x}^{Mp}})^{M}\quad(\mbox{since }Mp\ge1)\\
 & =(\sum_{j}\|R_{j}\|_{L_{x}^{p}}^{1/M})^{M}.
\end{align*}
Similarly, we have $\|\sum_{j}|\na R_{j}|\|_{L_{x}^{p}}\le(\sum_{j}\|\na R_{j}\|_{L_{x}^{p}}^{1/M})^{M}$,
which implies the second inequality.
\end{proof}
To avoid cumbersome notation, we define 
\begin{equation}
\begin{aligned}A_{p} & =A_{p}(\{\oa_{j}\})=\Bigl(\sum_{j}\oa_{j}^{\min(1,p)(\fc 1{\apa_{1}}-\fc d{2p})}\Bigr)^{\max(1,p^{-1})},\\
B_{p} & =B_{p}(\{\oa_{j}\},\{v_{j}\})=\Bigl(\sum_{j}\langle v_{j}\rangle^{\min(1,p)}\oa_{j}^{\min(1,p)(\fc 1{\apa_{1}}-\fc d{2p})}\Bigr)^{\max(1,p^{-1})},
\end{aligned}
\label{eq:defnAB}
\end{equation}
for $0<p\le\infty$. By Lemma \ref{lem:moving} (with $M=\max(1,p^{-1})$),
we have 
\begin{equation}
\|\sum_{j}|R_{j}|\|_{L_{x}^{p}}\le D_p A_{p},\quad\mbox{and}\quad
\|\sum_{j}|\na R_{j}|\|_{L_{x}^{p}}\lesssim D_p B_{p}.
\label{eq3.4}
\end{equation}
In particular, $\|T\|_{L_{x}^{p}}\le D_p A_{p}$, and 
$\|\na T\|_{L_{x}^{p}}\lesssim D_p B_{p}$. 

As discussed in the Introduction, to construct solutions of the
form $T+\eta$, we consider the operator 
\begin{equation}
\Phi\eta(t)=-i\int_{t}^{\infty}e^{i(t-\tau)\Da}[G(\tau)+H(\tau)]\, d\tau,\label{eq:Phi_ST}
\end{equation}
where 
\[
G=f(T+\eta)-f(T),\quad\mbox{and}\quad H=f(T)-\sum_{j}f(R_{j}).
\]
Define
\begin{equation}
v_{*}\coloneqq\fc 12\inf_{j,k\in\mb N,\, j<k}
\min(1,\oa_{j}^{1/2},\oa_{k}^{1/2})|v_{j}-v_{k}|.\label{eq:vstar}
\end{equation}
The following lemma gives more precise and complete estimates of $H$
than those given in \cite[Lemma 4.2, Lemma 4.4]{LeCoz_Tsai}. 
\begin{lem}
\label{lem:imp} We have the following estimates for the source term $H$:
\begin{itemize}
\item [\textup{(H0)}] Fix any $r_0>0$. 
For $r>s>r_{0}$ and $t\ge0$, 
\[
\|H(t)\|_{L_{x}^{r}}\lesssim_{r_0}(\sum_{i=1,2}A_{(\apa_{i}+1)s}^{\apa_{i}+1})^{s/r}(\sum_{i=1,2}A_{\infty}^{\apa_{i}+1})^{1-s/r}e^{-a(1-s/r)v_{*}t}.
\]

\item [\textup{(H1)}] Fix any $r_{1}>0$. 
For $r>s>r_{1}$ and $t\ge0$,
\[
\|\na H(t)\|_{L_{x}^{r}}\lesssim_{r_1}(\sum_{i=1,2}A_{\apa_{i}q}^{\apa_{i}}B_{p})^{s/r}(\sum_{i=1,2}A_{\infty}^{\apa_{i}}B_{\infty})^{1-s/r}e^{-a\min(\apa_{1},1)(1-s/r)v_{*}t},
\]
where $p,q$ are arbitrary numbers in $(0,\infty]$ satisfying $\fc 1q+\fc 1p=\fc 1s$.
\end{itemize}
\end{lem}

\begin{rem}
The inequalities are indeed true for all $r>s>0$, 
only that the multiplicative constants will then 
depend on $s$. For the upper bounds given in (H0) and (H1) to be 
under desirable control, there are actually natural
choices of $r_{0}$ and $r_{1}$ that are universal 
(depending only on $d,\apa_1,\apa_2$). We'll discuss 
this point right after the proof.
\end{rem}

\begin{proof}[Proof of Lemma \ref{lem:imp}]
Each assertion is proved by the same strategy as in \cite{LeCoz_Li_Tsai,LeCoz_Tsai}:
Prove the exponential decay in $t$ of the $L_{x}^{\infty}$ norm,
by singling out the soliton ``nearest'' to a fixed $(x,t)$. And
prove the boundedness of the $L_{x}^{s}$ norm independent of $t$.
Then the $L_{x}^{r}$ estimate follows by interpolation. 

\textsc{Proof of (H0).} For fixed $t,x$, let $m=m(t,x)\in\mb N$
be such that 
\[
|x-v_{m}t|=\min_{j\in\mb N}|x-v_{j}t|.
\]
Then for $j\ne m$, 
\[
|x-v_{j}t|=|x-v_{m}t+(v_{m}-v_{j})t|\ge|v_{j}-v_{m}|t-|x-v_{m}t|\ge|v_{j}-v_{m}|t-|x-v_{j}t|,
\]
and hence 
\begin{equation}
|x-v_{j}t|\ge\fc 12|v_{j}-v_{m}|t.
\label{eq3.7}
\end{equation}
By \eqref{eq:fineq0},
\begin{align*}
|H| & \le|f(T)-f(R_{m})|+\sum_{j\ne m}|f(R_{j})|\\
 & \lesssim\sum_{i=1,2}\Bigl\{|T-R_{m}|(|R_{m}|+|T-R_{m}|)^{\apa_{i}}+\sum_{j\ne m}|R_{j}|^{\apa_{i}+1}\Bigr\}\\
 & \le\sum_{i=1,2}\Bigl\{(\sum_{j\ne m}|R_{j}|)(\sum_{j}|R_{j}|)^{\apa_{i}}+(\sum_{j\ne m}|R_{j}|)^{\apa_{i}+1}\Bigr\}.
\end{align*}
Thus, by \eqref{eq:raw} and the definition of $v_{*}$, 
\begin{align*}
|H| & \lesssim\sum_{i=1,2}\Bigl\{(\sum_{j\ne m}\oa_{j}^{\fc 1{\apa_{1}}}e^{-av_{*}t})(\sum_{j}\oa_{j}^{\fc 1{\apa_{1}}})^{\apa_{i}}+(\sum_{j\ne m}\oa_{j}^{\fc 1{\apa_{1}}}e^{-av_{*}t})^{\apa_{i}+1}\Bigr\}\\
 & \lesssim\sum_{i=1,2}\Bigl\{(\sum_{j\ne m}\oa_{j}^{\fc 1{\apa_{1}}})(\sum_{j}\oa_{j}^{\fc 1{\apa_{1}}})^{\apa_{i}}+(\sum_{j\ne m}\oa_{j}^{\fc 1{\apa_{1}}})^{\apa_{i}+1}\Bigr\} e^{-av_{*}t}\quad(t\ge0)\\
 & \lesssim(\sum_{i=1,2}A_{\infty}^{\apa_{i}+1})e^{-av_{*}t}.
\end{align*}
Now that the upper bound is independent of $x$ and $m$, we get
\begin{equation}
\|H(t)\|_{L_{x}^{\infty}}
\lesssim(\sum_{i=1,2}A_{\infty}^{\apa_{i}+1})e^{-av_{*}t}.\label{eq:H_infty}
\end{equation}
Next, we try to bound $\|H\|_{L_{x}^{s}}$ for finite $s>r_{0}>0$. 
By Proposition \ref{prop:decomp}, in particular its flexibility of 
choosing $\theta_{ij}$ and $\phi_{ij}$,
\begin{align*}
|H| & \lesssim\sum_{i=1,2}\sum_{j}\Bigl\{|R_{j}|^{\max(\apa_{i},1)}(\sum_{\ell\ne j}|R_{\ell}|)^{\min(\apa_{i},1)}+|R_{j}|(\sum_{\ell\ne j}|R_{\ell}|)^{\apa_{i}}\Bigr\}\\
 & \le\sum_{i=1,2}\sum_{j}\Bigl\{|R_{j}|^{\max(\apa_{i},1)}(\sum_{\ell}|R_{\ell}|)^{\min(\apa_{i},1)}+|R_{j}|(\sum_{\ell}|R_{\ell}|)^{\apa_{i}}\Bigr\}\\
 & \le\sum_{i=1,2}\Bigl\{(\sum_{j}|R_{j}|^{\max(\apa_{i},1)})(\sum_{\ell}|R_{\ell}|)^{\min(\apa_{i},1)}+(\sum_{j}|R_{j}|)(\sum_{\ell}|R_{\ell}|)^{\apa_{i}}\Bigr\}.
\end{align*}
Since $\sum_{j}|R_{j}|^{\max(\apa_{i},1)}\le(\sum_{j}|R_{j}|)^{\max(\apa_{i},1)}$ due to $\max(\apa_{i},1)\ge1$,
we get 
\[
|H|\lesssim\sum_{i=1,2}(\sum_{j}|R_{j}|)^{\apa_{i}+1}.
\]
Thus, for $s>r_{0}$, by \eqref{eq3.4} (and Remark \ref{rem:p_lb})
\begin{equation}
\|H\|_{L_{x}^{s}}\lesssim\sum_{i=1,2}\|\sum_{j}|R_{j}|\|_{L_{x}^{(\apa_{i}+1)s}}^{\apa_{i}+1}\lesssim_{r_0}
\sum_{i=1,2}A_{(\apa_{i}+1)s}^{\apa_{i}+1}.\label{eq:H_s}
\end{equation}
By \eqref{eq:H_infty} and \eqref{eq:H_s}, for $r>s>r_{0}$, we have
\[
\|H\|_{L_{x}^{r}}\le\|H\|_{L_{x}^{s}}^{s/r}\|H\|_{L_{x}^{\infty}}^{1-s/r}\lesssim_{r_0}
(\sum_{i=1,2}A_{(\apa_{i}+1)s}^{\apa_{i}+1})^{s/r}(\sum_{i=1,2}A_{\infty}^{\apa_{i}+1})^{1-s/r}e^{-a(1-s/r)v_{*}t}.
\]

\textsc{Proof of (H1).} By \eqref{eq:fchain} and \eqref{eq:fineq1},
\begin{align}
|\na H| & \le\sum_{j}(|f_{z}(T)-f_{z}(R_{j})|+|f_{\bar{z}}(T)-f_{\bar{z}}(R_{j})|)|\na R_{j}|\nonumber \\
 & \lesssim\sum_{i=1,2}\sum_{j}(|T-R_{j}|)^{\min(\apa_{i},1)}(|T-R_{j}|+|R_{j}|)^{\max(\apa_{i}-1,0)}|\na R_{j}|\nonumber \\
 & \le\sum_{i=1,2}\sum_{j}(\sum_{\ell\ne j}|R_{\ell}|)^{\min(\apa_{i},1)}(\sum_{\ell}|R_{\ell}|)^{\max(\apa_{i}-1,0)}|\na R_{j}|\label{eq:3rdline}\\
 & \lesssim\sum_{i=1,2}A_{\infty}^{\max(\apa_{i}-1,0)}E_{i},\nonumber 
\end{align}
where
\[
E_{i}:=\sum_{j}(\sum_{\ell\ne j}|R_{\ell}|)^{\min(\apa_{i},1)}|\na R_{j}|.
\]
Let $m=m(t,x)\in\mb N$ be as above. By \eqref{eq:raw} and \eqref{eq3.7}, 
\begin{align*}
E_{i} & \lesssim\sum_{j\ne m}(\sum_{\ell\ne j}|R_{\ell}|)^{\min(\apa_{i},1)}|\na R_{j}|+(\sum_{\ell\ne m}|R_{\ell}|)^{\min(\apa_{i},1)}|\na R_{m}|\\
 & \lesssim\sum_{j\ne m}A_{\infty}^{\min(\apa_{i},1)}|\na R_{j}|+(\sum_{\ell\ne m}|R_{\ell}|)^{\min(\apa_{i},1)}B_{\infty}\\
 & \lesssim A_{\infty}^{\min(\apa_{i},1)}B_{\infty}e^{-av_{*}t}+B_{\infty}(A_{\infty}e^{-av_{*}t})^{\min(\apa_{i},1)}\\
 & \le(2A_{\infty}^{\min(\apa_{i},1)}B_{\infty})e^{-a\min(\apa_{i},1)v_{*}t}.\quad(t\ge0)
\end{align*}
Thus
\begin{equation}
\|\na H(t)\|_{L_{x}^{\infty}}\lesssim(\sum_{i=1,2}A_{\infty}^{\apa_{i}}B_{\infty})e^{-a\min(\apa_{1},1)v_{*}t}.\label{eq:gH_infty}
\end{equation}
On the other hand, from \eqref{eq:3rdline}, 
\begin{align*}
|\na H| & \lesssim\sum_{i=1,2}\sum_{j}(\sum_{\ell}|R_{\ell}|)^{\min(\apa_{i},1)}(\sum_{\ell}|R_{\ell}|)^{\max(\apa_{i}-1,0)}|\na R_{j}|\\
 & =\sum_{i=1,2}(\sum_{\ell}|R_{\ell}|)^{\apa_{i}}(\sum_{j}|\na R_{j}|).
\end{align*}
Hence, for $s>r_{1}$,
\begin{equation}
\|\na H\|_{L_{x}^{s}}
\lesssim\sum_{i=1,2}\|(\sum_{\ell}|R_{\ell}|)^{\apa_{i}}\|_{L_{x}^{q}}
\|\sum_{j}|\na R_{j}|\|_{L^{p}}
\lesssim_{r_1}\sum_{i=1,2}A_{\apa_{i}q}^{\apa_{i}}B_{p},\label{eq:gH_s}
\end{equation}
where $p,q$ are any numbers in $(0,\infty]$ satisfying $\fc 1q+\fc 1p=\fc 1s$.
By \eqref{eq:gH_infty} and \eqref{eq:gH_s}, for $r>s>r_{1}$, we
have
\[
\|\na H(t)\|_{L_{x}^{r}}\lesssim_{r_1}
(\sum_{i=1,2}A_{\apa_{i}q}^{\apa_{i}}B_{p})^{s/r}(\sum_{i=1,2}A_{\infty}^{\apa_{i}}B_{\infty})^{1-s/r}e^{-a\min(\apa_{1},1)(1-s/r)v_{*}t}.\qedhere
\]
\end{proof}

Now we explain how the values
of $A_{p},B_{p}$ (here $p$ is regarded as a parameter) and $v_{*}$
should be controlled, by adjusting $\{\oa_{j}\}$ and $\{v_{j}\}$
of the profile $T$. As is mentioned, we need $\oa_{j}\to0$ and 
$|v_{j}-v_{k}|\to\infty$. Precisely, \emph{we will need the flexibility
of making $v_{*}$ as large as we like, and at the same time controlling
the sizes of $A_{p}$ and $B_{p}$.} As to this purpose, the first
obvious observation is that $A_{p}<\infty$ can be true if and only
if $\fc 1{\apa_{1}}-\fc d{2p}>0$, i.e. $p>\fc{d\apa_{1}}2$. Next,
a little thought shows that $v_{*}>0$ and $B_{p}<\infty$ can hold
simultaneously only if $\fc 1{\apa_{1}}-\fc d{2p}>\fc 12$, which
is equivalent to $\apa_{1}<2$ and $p>\fc{d\apa_{1}}{2-\apa_{1}}$.
It turns out that these minimum requirements are sufficient. We give
the relevant facts in the next lemma. For convenience, we define 
\begin{equation}
\label{eq:CACB.def}
\ml C_{A}=(\fc{d\apa_{1}}2,\infty]; \quad 
\ml C_{B}=(\fc{d\apa_{1}}{2-\apa_{1}},\infty] \quad (\text{if } \apa_{1}<2).
\end{equation}

\begin{lem}
\label{lem:compete}~
\begin{itemize}
\item[\textup{(a)}]  For $\infty\ge q>p\in\ml C_{A}$, we have 
$A_{q}<\max(1,\oa_{*})^{\fc 1{\apa_{1}}}A_{p}$ whenever 
$A_{p}<\infty$. If $\apa_{1}<2$ and $\infty\ge q>p\in\ml C_{B}$,
we have $B_{q}<\max(1,\oa_{*})^{\fc 1{\apa_{1}}}B_{p}$ whenever 
$B_{p}<\infty$.

\item[\textup{(b)}] Suppose $q\in\ml C_{A}$, then for any constants 
$c,\La>0$, there exist $\{\oa_{j}\}$ and $\{v_{j}\}$ such that 
$A_{q}\le c$, and $v_{*}\ge\La$. If moreover $\apa_{1}<2$ and 
$p\in\ml C_{B}$, then $\{\oa_{j}\}$ and $\{v_{j}\}$
can be chosen so that $B_{p}\le c$ is also satisfied.

\end{itemize}
\end{lem}

The proofs of these facts are elementary and are given in 
Appendix \ref{appI}.
Briefly, (a) says $A_{q}\lesssim A_{p}$ and $B_{q}\lesssim B_{p}$
for $q\ge p$. As a consequence, when there are several $A_{p}$ or
$B_{p}$ to be controlled, it suffices to control those having smaller
$p$. And (b) is exactly the desired control. (a) and (b)
will be fundamental for the effectiveness of our estimates of $G$
and $H$. 

For the construction of soliton trains in this section, the needed estimates will be derived from
the dispersive inequality: If $p\in[2,\infty]$ and $t\ne0$, 
\begin{equation}
\|e^{it\Da}u\|_{L^{p}(\mb R^{d})}\le(4\pi|t|)^{-d(\fc 12-\fc 1p)}\|u\|_{L^{p'}(\mb R^{d})}\quad\forall\, u\in L^{p'}(\mb R^{d}).\label{eq:dispersive}
\end{equation}
We now give our first main result. 

\begin{thm}
\label{thm:single1} Let $d=1$, and $f$  
satisfy Assumptions \ref{assuF} and \ref{assuT}. 
Suppose moreover $\apa_{1}\ge 1$. Then for any finite $\rho>0$, 
there is a constant $\la_{0}>0$ such that
the following holds: For $\la_{0}\le\la<\infty$, there exist 
solutions of \eqref{eq:NLSE} of the form \eqref{eq:T-def} 
for $t \ge 0$, with
\begin{equation}
\sup_{t\ge 0}e^{\la t}\|\eta(t)\|_{L_{x}^{2}\cap L_{x}^{\infty}}\le\rho.\label{eq:Thm1_norm}
\end{equation}
\end{thm}
\begin{proof}
For $0<\la<\infty$, let $X=X_{\la}$ be the Banach space of all $\eta:[0,\infty)\times\mb R^{1}\to\mb C$
with norm $\|\eta\|_{X}$ defined by the left-hand side of \eqref{eq:Thm1_norm}.
By interpolation, we have
\begin{equation}
\|\eta(t)\|_{L_{x}^{p}}\le\|\eta\|_{X}e^{-\la t}\quad\forall\, p\in[2,\infty], \quad
\forall t\ge 0.
\label{eq:Thm1_e}
\end{equation}

Given $\rho\in(0,\infty)$, we will prove that, for sufficiently large
$\la$, there are $\{\oa_{j}\},\{v_{j}\}$ such that $\Phi$ (defined
in \eqref{eq:Phi_ST}) is a contraction mapping on the closed ball
$\{\eta\in X:\|\eta\|_{X}\le\rho\}$.

First, we give estimates for $\Phi$ to be a self-mapping. Given $\eta\in X$
with $\|\eta\|_{X}\le\rho$. For $p\in[2,\infty]$, the dispersive
inequality \eqref{eq:dispersive} implies 
\[
\|\Phi\eta(t)\|_{L_{x}^{p}}\lesssim\int_{t}^{\infty}|t-\tau|^{-(\fc 12-\fc 1p)}(\|G(\tau)\|_{L_{x}^{p'}}+\|H(\tau)\|_{L_{x}^{p'}})\, d\tau.
\]
To estimate $\|\Phi\eta\|_{X}$, we have to estimate $\|G(\tau)\|_{L_{x}^{2}}$,
$\|G(\tau)\|_{L_{x}^{1}}$, $\|H(\tau)\|_{L_{x}^{2}}$ and $\|H(\tau)\|_{L_{x}^{1}}$. 

By \eqref{eq:fineq0}, 
\[
|G|=|f(T+\eta)-f(T)|\lesssim\sum_{i=1,2}\left\{ |\eta||T|^{\apa_{i}}+|\eta|^{\apa_{i}+1}\right\} .
\]
For the first term, we have 
\begin{align}
\||\eta||T|^{\apa_{i}}(\tau)\|_{L_{x}^{2}} & \le\|\eta(\tau)\|_{L_{x}^{2}}\|T\|_{L_{x}^{\infty}}^{\apa_{i}}\lesssim\rho A_{\infty}^{\apa_{i}}e^{-\la\tau},\label{eq:Thm_G_1_2}\\
\||\eta||T|^{\apa_{i}}(\tau)\|_{L_{x}^{1}} & \le\|\eta(\tau)\|_{L_{x}^{2}}\|T\|_{L_{x}^{2\apa_{i}}}^{\apa_{i}}\lesssim\rho A_{2\apa_{i}}^{\apa_{i}}e^{-\la\tau},\label{eq:Thm_G_1_1}
\end{align}
where notice that $2\apa_{1}\in\ml C_{A}$. For the second term, by
\eqref{eq:Thm1_e}, 
\begin{align}
\||\eta|^{\apa_{i}+1}(\tau)\|_{L_{x}^{2}} & =\|\eta(\tau)\|_{L_{x}^{2(\apa_{i}+1)}}^{\apa_{i}+1}\le\rho^{\apa_{i}+1}e^{-(\apa_{i}+1)\la\tau}\le\rho^{\apa_{i}+1}e^{-\la\tau},\label{eq:Thm_G_2_2}\\
\||\eta|^{\apa_{i}+1}(\tau)\|_{L_{x}^{1}} & =\|\eta(\tau)\|_{L_{x}^{\apa_{i}+1}}^{\apa_{i}+1}\le\rho^{\apa_{i}+1}e^{-(\apa_{i}+1)\la\tau}\le\rho^{\apa_{i}+1}e^{-\la\tau},\label{eq:Thm_G_2_1}
\end{align}
where in \eqref{eq:Thm_G_2_1} we use the assumption $\apa_{1}\ge1$. 

For $H$, taking $r=2$ and $s=1$ in Lemma \ref{lem:imp} (H0),
we get
\begin{equation}
\|H(\tau)\|_{L_{x}^{2}}\lesssim(\sum_{i=1,2}A_{\apa_{i}+1}^{\apa_{i}+1})^{1/2}(\sum_{i=1,2}A_{\infty}^{\apa_{i}+1})^{1/2}e^{-\fc a2v_{*}\tau},\label{eq:Thm1_H_2}
\end{equation}
with $\apa_{1}+1\in\ml C_{A}$. Taking $r=1$ and $s=1/2$,
we get 
\begin{equation}
\|H(\tau)\|_{L_{x}^{1}}\lesssim(\sum_{i=1,2}A_{(\apa_{i}+1)/2}^{\apa_{i}+1})^{1/2}(\sum_{i=1,2}A_{\infty}^{\apa_{i}+1})^{1/2}e^{-\fc a2v_{*}\tau},\label{eq:Thm1_H_1}
\end{equation}
with $(\apa_{1}+1)/2\in\ml C_{A}$.

Now suppose 
\begin{equation}
\fc a2v_{*}\ge\la.\label{eq:Thm_1_vstar}
\end{equation}
Then from \eqref{eq:Thm_G_1_2}, \eqref{eq:Thm_G_2_2} and \eqref{eq:Thm1_H_2},
we get 
\begin{equation}
\|\Phi\eta(t)\|_{L_{x}^{2}}\lesssim\int_{t}^{\infty}E_{1}e^{-\la\tau}\, d\tau=E_{1}\la^{-1}e^{-\la t},\label{eq:Thm1Phi2}
\end{equation}
where 
\[
E_{1}=\sum_{i=1,2}(\rho A_{\infty}^{\apa_{i}}+\rho^{\apa_{i}+1})+(\sum_{i=1,2}A_{\apa_{i}+1}^{\apa_{i}+1})^{1/2}(\sum_{i=1,2}A_{\infty}^{\apa_{i}+1})^{1/2}.
\]
From \eqref{eq:Thm_G_1_1}, \eqref{eq:Thm_G_2_1} and \eqref{eq:Thm1_H_1},
we get
\begin{equation}
\|\Phi\eta(t)\|_{L_{x}^{\infty}}\lesssim\int_{t}^{\infty}|t-\tau|^{-\fc 12}E_{2}e^{-\la\tau}\, d\tau=E_{2}\Ga(1/2)\la^{-1/2}e^{-\la t},\label{eq:Thm1Phiinfty}
\end{equation}
where 
\[
E_{2}=\sum_{i=1,2}(\rho A_{2\apa_{i}}^{\apa_{i}}+\rho^{\apa_{i}+1})+(\sum_{i=1,2}A_{(\apa_{i}+1)/2}^{\apa_{i}+1})^{1/2}(\sum_{i=1,2}A_{\infty}^{\apa_{i}+1})^{1/2},
\]
and $\Ga(z)=\int_{0}^{\infty}x^{z-1}e^{-x}dx$ is the Gamma function. 

Next, we give estimates for contractivity. Given $\eta_{1},\eta_{2}\in X$,
$\|\eta_{1}\|_{X},\|\eta_{2}\|_{X}\le\rho$. We have 
\[
\Phi\eta_{1}-\Phi\eta_{2}=-i\int_{t}^{\infty}e^{i(t-\tau)\Da}[f(T+\eta_{1})-f(T+\eta_{2})](\tau)\, d\tau.
\]
Hence, for $p\in[2,\infty]$, 
\[
\|(\Phi\eta_{1}-\Phi\eta_{2})(t)\|_{L_{x}^{p}}\lesssim\int_{t}^{\infty}|t-\tau|^{-(\fc 12-\fc 1p)}\|[f(T+\eta_{1})-f(T+\eta_{2})](\tau)\|_{L_{x}^{p'}}\, d\tau.
\]
By \eqref{eq:fineq0}, 
\begin{align*}
|f(T+\eta_{1})-f(T+\eta_{2})| & \lesssim\sum_{i=1,2}\Bigl\{|\eta_{1}-\eta_{2}||T+\eta_{2}|^{\apa_{i}}+|\eta_{1}-\eta_{2}|^{\apa_{i}+1}\Bigr\}\\
 & \lesssim\sum_{i=1,2}\left\{ |\eta_{1}-\eta_{2}||T|^{\apa_{i}}+|\eta_{1}-\eta_{2}|(|\eta_{1}|+|\eta_{2}|)^{\apa_{i}}\right\} .
\end{align*}
By \eqref{eq:Thm1_e}, 
\begin{align*}
\||\eta_{1}-\eta_{2}||T|^{\apa_{i}}(\tau)\|_{L_{x}^{2}} & \le\|(\eta_{1}-\eta_{2})(\tau)\|_{L_{x}^{2}}\|T\|_{L_{x}^{\infty}}^{\apa_{i}}\\
 & \lesssim A_{\infty}^{\apa_{i}}\|\eta_{1}-\eta_{2}\|_{X}e^{-\la\tau},
\end{align*}
\begin{align*}
\||\eta_{1}-\eta_{2}||T|^{\apa_{i}}(\tau)\|_{L_{x}^{1}} & \le\|(\eta_{1}-\eta_{2})(\tau)\|_{L_{x}^{2}}\|T(\tau)\|_{L_{x}^{2\apa_{i}}}^{\apa_{i}}\\
 & \lesssim A_{2\apa_{i}}^{\apa_{i}}\|\eta_{1}-\eta_{2}\|_{X}e^{-\la\tau},
\end{align*}
\begin{align*}
\||\eta_{1}-\eta_{2}|(|\eta_{1}|+|\eta_{2}|)^{\apa_{i}}(\tau)\|_{L_{x}^{2}} & \le\|(\eta_{1}-\eta_{2})(\tau)\|_{L_{x}^{2}}\|(|\eta_{1}|+|\eta_{2}|)(\tau)\|_{L_{x}^{\infty}}^{\apa_{i}}\\
 & \le(2\rho)^{\apa_{i}}\|\eta_{1}-\eta_{2}\|_{X}e^{-\la\tau},
\end{align*}
\begin{align*}
\||\eta_{1}-\eta_{2}|(|\eta_{1}|+|\eta_{2}|)^{\apa_{i}}(\tau)\|_{L_{x}^{1}} & \le\|(\eta_{1}-\eta_{2})(\tau)\|_{L_{x}^{2}}\|(|\eta_{1}|+|\eta_{2}|)(\tau)\|_{L_{x}^{2\apa_{i}}}^{\apa_{i}}\\
 & \le(2\rho)^{\apa_{i}}\|\eta_{1}-\eta_{2}\|_{X}e^{-\la\tau}.
\end{align*}
Hence
\begin{equation}
\|(\Phi\eta_{1}-\Phi\eta_{2})(t)\|_{L_{x}^{2}}\lesssim E_{3}\la^{-1}\|\eta_{1}-\eta_{2}\|_{X}e^{-\la t},\label{eq:Thm1D2}
\end{equation}
where $E_{3}=\sum_{i=1,2}(A_{\infty}^{\apa_{i}}+(2\rho)^{\apa_{i}})$.
And
\begin{equation}
\|(\Phi\eta_{1}-\Phi\eta_{2})(t)\|_{L_{x}^{\infty}}\lesssim E_{4}\Ga(1/2)\la^{-1/2}\|\eta_{1}-\eta_{2}\|_{X}e^{-\la t},\label{eq:Thm1Dinfty}
\end{equation}
where 
$E_{4}=\sum_{i=1,2}(A_{2\apa_{i}}^{\apa_{i}}+(2\rho)^{\apa_{i}})$. 

Now for any $\la>0$, Lemma \ref{lem:compete} ensures that we can 
choose $\{\oa_{j}\}$ and $\{v_{j}\}$ (depending on $\la$) such that 
$v_*$ satisfies \eqref{eq:Thm_1_vstar}, with all ``$A_p$'' being no 
larger than any preassigned number, say $A_p\le 1$. In particular,
we see there is a constant $E=E(\apa_1,\apa_2,\rho)>0$ such that 
$E_{\ell}\le E$ for $\ell =1,2,3,4$, given in \eqref{eq:Thm1Phi2},
\eqref{eq:Thm1Phiinfty}, \eqref{eq:Thm1D2}, and \eqref{eq:Thm1Dinfty}.
Thus, also from these inequalities, if $\la$ is large enough 
(i.e. $\la\in [\la_0,\infty)$ for some large enough $\la_0$), such 
choice of $\{\oa_j\},\{v_j\}$ gives
\begin{align*}
\|\Phi\eta(t)\|_{L^2_x\cap L^\infty_x}&\le \rho e^{-\la t}\\
\|(\Phi\eta_1-\Phi\eta_2)(t)\|_{L^2_x\cap L^\infty_x}
&\le \fc{1}{2}\|\eta_1-\eta_2\|_X e^{-\la t}.
\end{align*}
Hence $\Phi$ is a contraction mapping on the closed ball of $X$
with radius $\rho$.
\end{proof}

\begin{rem}
By the contraction mapping principle, for a fixed profile $T$ such
that $\Phi$ is a contraction, the error $\eta$ is unique within
the class we try to find it.
\end{rem}

Before giving the next theorem, we make some comments on 
the choices of $\{\oa_j\}$ and $\{v_j\}$.
As gradient estimate of $\eta$ is not involved in the 
previous proof, $B_p$ does not occur, and the last part of 
the proof can be replaced by the following:
1) First choose $\{\oa_j\}$ so that the coefficients
$E_1$ $\sim$ $E_4$ are finite 
(equivalently, all ``$A_p$'' are finite), then
2) choose $\la\ge \la_0$ sufficiently large so that 
\eqref{eq:Thm1Phi2} -- \eqref{eq:Thm1Dinfty} imply that 
$\Phi$ is a contraction mapping. And hence 3) the construction 
is done for any $\{v_j\}$ such that \eqref{eq:Thm_1_vstar} 
is satisfied.

It's then easy to see what choices of $\{\oa_j\},\{v_j\}$ 
are allowable. For example, since $A_{(\apa_{1}+1)/2}$ is 
the $A_{p}$ with smallest $p$ to be controlled in the proof, 
the construction is possible if and only if $\{\oa_j\}$ is 
such that $A_{(\apa_{1}+1)/2}<\infty$, i.e.
\begin{equation}\label{eq:sample}
\sum_{j}\oa_{j}^{\fc 1{\apa_{1}}-\fc 1{\apa_{1}+1}}<\infty.
\end{equation}

However, when there is $B_p$, the Step 3) of choosing 
$\{v_j\}$ will also influence the coefficients considered in 
Step 1). For later considerations,
we have given a proof that works even when $B_p$ is present:  
For every $\la$, Lemma \ref{lem:compete} ensures that
we can choose $\{\oa_j\}$ and $\{v_j\}$ so that $v_*$ is large 
enough and all $A_p, B_p$ are small. For large enough $\la$,
$\Phi$ is hence a contraction mapping for such 
$\{\oa_j\},\{v_j\}$. 
Moreover, giving precise
conditions as \eqref{eq:sample}, though possible, would 
be rather cumbersome. 
We shall hence satisfy ourselves 
with such vague statement as Theorem \ref{thm:single1}. 
Suffice it to say that, once a construction is done with 
some choice of $\{\oa_{j}\}$ and $\{v_{j}\}$, it is done 
with all other choices making the present $A_{p}$ and 
$B_{p}$ smaller and the $v_{*}$ larger. One easy way to 
obtain such ``better'' choices is by rescaling, i.e. by 
considering $\{\ka \oa_{j}\}$ and $\{\nu v_{k}\}$ for 
suitable positive constants $\ka,\nu$. The argument is 
routine and we omit the details.

We now turn to our next main result. First, notice that
the proof of Theorem \ref{thm:single1} fails for $d\ge 2$, 
since the dispersive inequality gives 
\[
\|\Phi\eta(t)\|_{L_{x}^{\infty}}\lesssim\int_{t}^{\infty}|t-\tau|^{-\fc d2}(\|G(\tau)\|_{L_{x}^{1}}+\|H(\tau)\|_{L_{x}^{1}})\, d\tau,
\]
where the singularity at $\tau=t$ is not integrable. 
As a consequence, we consider the following alternative way: 
Construct trains $T+\eta$ having $\|\eta(t)\|_{L_{x}^{2}}$ and 
$\|\na\eta(t)\|_{L_{x}^{r}}$ controls for some $r>d$. Then the 
$\|\eta(t)\|_{L_{x}^{\infty}}$ control follows from Sobolev 
embedding (Gagliardo-Nirenberg's inequality).
It turns out that we still need $d\le 3$. Moreover, due to some 
technical benefits, we also assume the 
$\|\na\eta(t)\|_{L_{x}^{2}}$ control in our construction 
(see Remark \ref{rem:needofg2} after the proof).

\begin{thm}
\label{thm:single2} Let $d\le 3$, and $f$ 
satisfy Assumptions \ref{assuF} and \ref{assuT}.
Suppose $2(\fc 12-\fc 1d)<\apa_{1}<2$
(the lower bound is empty unless $d=3$). 
Then for any finite $\rho>0$,
there are constants $r>d$, $\la_{0}>0$, and $0<c_{1}\le1$, such
that the following holds: For $\la_{0}\le\la<\infty$, there exist
solutions of \eqref{eq:NLSE}
of the form \eqref{eq:T-def}, with
\begin{equation}
\sup_{t\ge0}\bigl\{ e^{\la t}\|\eta(t)\|_{L_{x}^{2}}+e^{c_{1}\la t}\|\na\eta(t)\|_{L_{x}^{2}\cap L_{x}^{r}}\bigr\}\le\rho.\label{eq:Thm2_norm}
\end{equation}
If moreover $d=1$ and $\apa_{1}\ge1$, then the above assertion holds
with $r=\infty$.\end{thm}
{\it Remark.}  We need $d \le 3$ so that 
$\| e^{it\Delta}\| _{L^{r'} \to L^r}$ is locally integrable 
in $t$ for some $r>d$.
We need $ \alpha _1 <2 $ so that $\ml C_B$ in \eqref{eq:CACB.def}
is nonempty, and hence $B_p$ can be controlled for $p \in \ml C_B$.

\begin{proof}
For $r>d$, $\la>0$, and $0<c_{1}\le1$, let $X=X_{r,\la,c_{1}}$
be the Banach space of all $\eta:[0,\infty)\times\mb R^{d}\to\mb C$
with norm $\|\eta\|_{X}$ defined by the left-hand side of \eqref{eq:Thm2_norm}.
By the Gagliardo-Nirenberg's inequality, for any $p\in[2,\infty]$,
\begin{equation}
\|\eta(t)\|_{L_{x}^{p}}\le G_{d,p,r}\|\eta(t)\|_{L_{x}^{2}}^{1-\ta}\|\na\eta(t)\|_{L_{x}^{r}}^{\ta}\le G_{d,p,r}\|\eta\|_{X}e^{-(1-\ta+c_{1}\ta)\la t},\label{eq:lbest}
\end{equation}
where $G_{d,p,r}$ is a constant and $\ta=(\fc 12-\fc 1p)(\fc 12+\fc 1d-\fc 1r)^{-1}$.

We will show that $\Phi$ can be a contraction mapping on the closed
unit ball of $X$ (the case of $\rho=1$). Balls with other radius can be 
similarly treated.
Moreover, we'll only give the estimates for $\Phi$ to be a self-mapping.
As in the proof of Theorem \ref{thm:single1}, the derivations of the estimates for contractivity are no harder (and without the $H$ parts).

Given $\eta\in X$ with $\|\eta\|_{X}\le1$. We will first estimate
$\|\Phi\eta(t)\|_{L_{x}^{2}}$, and then $\|\na\Phi\eta(t)\|_{L_{x}^{r}}$.
Finally, $\|\na\Phi\eta(t)\|_{L_{x}^{2}}$ is basically a special
case of $\|\na\Phi\eta(t)\|_{L_{x}^{r}}$.

\medskip

\textbf{Part 1. Estimate of $\|\Phi\eta(t)\|_{L_{x}^{2}}$.} For $G$,
we have
\begin{align}
\|G(\tau)\|_{L_{x}^{2}} & \lesssim\sum_{i=1,2}(\||\eta||T|^{\apa_{i}}(\tau)\|_{L_{x}^{2}}+\||\eta|^{\apa_{i}+1}(\tau)\|_{L_{x}^{2}})\nonumber \\
 & \le\sum_{i=1,2}(\|\eta(\tau)\|_{L_{x}^{2}}\|T\|_{L_{x}^{\infty}}^{\apa_{i}}+\|\eta(\tau)\|_{L_{x}^{2}}\|\eta(\tau)\|_{L_{x}^{\infty}}^{\apa_{i}})\nonumber \\
 & \lesssim\sum_{i=1,2}(A_{\infty}^{\apa_{i}}+G_{d,\infty,r}^{\apa_{i}})e^{-\la\tau},\label{eq:Thm2_G_2}
\end{align}
by \eqref{eq:lbest}. Then consider $H$. Since $d\le3$, we have
$2>\fc{d\apa_{1}}{2(\apa_{1}+1)}$ (for all $\apa_{1}>0$). Fix any
$2>s_{1}>\fc{d\apa_{1}}{2(\apa_{1}+1)}$, we get from Lemma \ref{lem:imp}
(H0) 
\begin{equation}
\|H(\tau)\|_{L_{x}^{2}}\lesssim(\sum_{i=1,2}A_{(\apa_{i}+1)s_{1}}^{\apa_{i}+1})^{s_{1}/2}(\sum_{i=1,2}A_{\infty}^{\apa_{i}+1})^{1-s_{1}/2}e^{-a(1-s_{1}/2)v_{*}\tau},\label{eq:Thm2_H_2}
\end{equation}
with $(\apa_{1}+1)s_{1}\in\ml C_{A}$. Suppose
\[
a(1-s_{1}/2)v_{*}\ge\la.
\]
Then from \eqref{eq:Thm2_G_2} and \eqref{eq:Thm2_H_2}, the dispersive
inequality gives 
\begin{equation}
\|\Phi\eta(t)\|_{L_{x}^{2}}\lesssim\int_{t}^{\infty}E_{1}e^{-\la\tau}\, d\tau=E_{1}\la^{-1}e^{-\la t},\label{eq:Thm2_L2}
\end{equation}
where 
\[
E_{1}=\sum_{i=1,2}(A_{\infty}^{\apa_{i}}+G_{d,\infty,r}^{\apa_{i}})+(\sum_{i=1,2}A_{(\apa_{i}+1)s_{1}}^{\apa_{i}+1})^{s_{1}/2}(\sum_{i=1,2}A_{\infty}^{\apa_{i}+1})^{1-s_{1}/2}.
\]

\textbf{Part 2. Estimate of $\|\na\Phi\eta(t)\|_{L_{x}^{r}}$.} This
part is more delicate. The dispersive inequality gives 
\[
\|\na\Phi\eta(t)\|_{L_{x}^{r}}\lesssim\int_{t}^{\infty}|t-\tau|^{-d(\fc 12-\fc 1r)}(\|\na G(\tau)\|_{L_{x}^{r'}}+\|\na H(\tau)\|_{L_{x}^{r'}})\, d\tau.
\]
To derive a suitable estimate from it, in the following we will get
several conditions on the lower bounds of $1/r$. The one thing to
check is that they are all strictly less than $1/d$, so that there
is really one $r>d$ satisfying all the conditions. Moreover, if $d=1$,
$r$ can be $\infty$. 

\emph{Step 1.} For $|t-\tau|^{-d(\fc 12-\fc 1r)}$ to be integrable
at the singularity $\tau=t$, we need $d(\fc 12-\fc 1r)<1$, i.e.
\[
\fc 12-\fc 1d<\fc 1r.\tag{Condition 1}
\]
Since we want $r>d$, we need the lower bound $\fc 12-\fc 1d$ to be less than $\fc 1d$,
which holds
if and only if $d\le3$.
If $d=1$, the lower bound is negative and we can choose $r=\infty$.

\emph{Step 2. Estimate of $G$.} By \eqref{eq:fineq1'},
\begin{align*}
|\na G| & =|\na[f(T+\eta)-f(T)]|\\
 & \lesssim\sum_{i=1,2}\left\{ |\eta|^{\min(\apa_{i},1)}(|T|+|\eta|)^{\max(\apa_{i}-1,0)}|\na T|+(|T|+|\eta|)^{\apa_{i}}|\na\eta|\right\} .
\end{align*}
If $\apa_{i}>1$, we have to estimate the $L_{x}^{r'}$ norm of (1)
$|\eta||T|^{\apa_{i}-1}|\na T|$, (2) $|\eta|^{\apa_{i}}|\na T|$,
(3) $|T|^{\apa_{i}}|\na\eta|$, and (4) $|\eta|^{\apa_{i}}|\na\eta|$.
And if $\apa_{i}\le1$, we only have to estimate (2), (3) and (4).
We discuss them in the following. We remark that the $p,q$ in different
sub-steps are unrelated.

\emph{Step 2-1. Estimate of (1). (Only for $\apa_{i}>1$)} Suppose
\begin{equation}
\fc 1{r'}<\fc 12+\fc{2(\apa_{i}-1)}{d\apa_{1}}+\fc 2d\left(\fc 1{\apa_{1}}-\fc 12\right).\label{eq:preC2}
\end{equation}
Then, since $\fc 12\le\fc 1{r'}$, we have $\fc 1{r'}=\fc 12+\fc 1q+\fc 1p$
for some $p,q\in(0,\infty]$ satisfying $\fc 1q<\fc{2(\apa_{i}-1)}{d\apa_{1}}$
and $\fc 1p<\fc 2d(\fc 1{\apa_{1}}-\fc 12)$. Thus
\begin{align}
\||\eta||T|^{\apa_{i}-1}|\na T|(\tau)\|_{L_{x}^{r'}} & \le\|\eta(\tau)\|_{L_{x}^{2}}\||T|^{\apa_{i}-1}\|_{L_{x}^{q}}\|\na T\|_{L_{x}^{p}}\nonumber \\
 & \lesssim A_{(\apa_{i}-1)q}^{\apa_{i}-1}B_{p}e^{-\la\tau}\nonumber \\
 & \le A_{(\apa_{i}-1)q}^{\apa_{i}-1}B_{p}e^{-c_{1}\la\tau},\label{eq:T2_P2_E1}
\end{align}
with $(\apa_{i}-1)q\in\ml C_{A}$ and $p\in\ml C_{B}$. Notice that
\eqref{eq:preC2} is equivalent to 
\[
\fc 12-\fc 2d\left(\fc{\apa_{i}}{\apa_{1}}-\fc 12\right)<\fc 1r.\tag{Condition 2}
\]
It's easy to check that, since $d\le3$ and $\apa_{i}\ge\apa_{1}$,
the lower bound is strictly less than $1/d$, and is negative if $d=1$.

\emph{Step 2-2. Estimate of (2).} Let $q=\max(r',2/\apa_{1})$, and
$p$ be such that $\fc 1{r'}=\fc 1q+\fc 1p$. By \eqref{eq:lbest}
\begin{align}
\||\eta|^{\apa_{i}}|\na T|(\tau)\|_{L_{x}^{r'}} & \le\|\eta(\tau)\|_{L_{x}^{\apa_{i}q}}^{\apa_{i}}\|\na T\|_{L_{x}^{p}}\nonumber \\
 & \lesssim B_{p}G_{d,\apa_{i}q,r}^{\apa_{i}}e^{-\apa_{i}(1-\ta+c_{1}\ta)\la\tau},\label{eq:T2_P2_E2}
\end{align}
where $\ta=(\fc 12-\fc 1{\apa_{i}q})(\fc 12+\fc 1d-\fc 1r)^{-1}$.
For \eqref{eq:T2_P2_E2} to be an effective estimate, we need (i)
$p\in\ml C_{B}$, and (ii) $\apa_{i}(1-\ta+c_{1}\ta)\ge c_{1}$. 

Since $\fc 1p=\fc 1{r'}-\fc 1q$, (i) holds iff 
\begin{equation}
\fc 1{r'}-\fc 1q<\fc 2d\left(\fc 1{\apa_{1}}-\fc 12\right).\label{eq:preC3}
\end{equation}
There are two cases according to the value of $q$. If $q=r'$ (i.e.
$r'\ge2/\apa_{1}$), then \eqref{eq:preC3} is automatically true
since $\apa_{1}<2$, and no restriction on $r$ is needed. On the
other hand, if $q=2/\apa_{1}$ (i.e. $r'<2/\apa_{1}$), then \eqref{eq:preC3}
gives
\[
1-\left(\fc{\apa_{1}}2+\fc 2{d\apa_{1}}\right)+\fc 1d<\fc 1r.\tag{Condition 3}
\]
Since 
\begin{equation}
\fc{\apa_{1}}2+\fc 2{d\apa_{1}}\ge\fc 2{\sqrt{d}},\label{eq:amgm}
\end{equation}
the lower bound is less than $1/d$ for $d\le3$. Moreover, if $d=1$,
strict inequality holds in \eqref{eq:amgm}, and the lower bound is
negative.

Now consider (ii). Since $r>d$, there exists some $c_{M}=c_{M}(d,\apa_{1},\apa_{2},r)>0$
such that (ii) holds as long as $0<c_{1}\le c_{M}$. For example,
we can use the (rather rough) estimate 
\[
\apa_{i}(1-\ta+c_{1}\ta)\ge\apa_{1}(1-\ta).
\]
Hence (ii) holds if
\begin{equation}
0<c_{1}\le\apa_{1}\left(1-\fc 12\Bigl(\fc 12+\fc 1d-\fc 1r\Bigr)^{-1}\right).\label{eq:c1cond}
\end{equation}

\emph{Step 2-3. Estimate of (3).} We have
\begin{equation}
\||T|^{\apa_{i}}|\na\eta|(\tau)\|_{L_{x}^{r'}}\le\||T|^{\apa_{i}}\|_{L_{x}^{q}}\|\na\eta(\tau)\|_{L_{x}^{2}}\lesssim A_{\apa_{i}q}^{\apa_{i}}e^{-c_{1}\la\tau},\label{eq:T2_P2_E3}
\end{equation}
where $\fc 1{r'}=\fc 1q+\fc 12$. We need $\apa_{i}q\in\ml C_{A}$,
i.e.
\[
\fc 12-\fc{2\apa_{i}}{d\apa_{1}}<\fc 1r.\tag{Condition 4}
\]
The lower bound is less than $1/d$, and is negative if $d=1$. 

\emph{Step 2-4. Estimate of (4).} We have 
\[
\||\eta|^{\apa_{i}}|\na\eta|(\tau)\|_{L_{x}^{r'}}\le\|\eta(\tau)\|_{L_{x}^{\apa_{i}q}}^{\apa_{i}}\|\na\eta(\tau)\|_{L_{x}^{2}},
\]
where $\fc 1{r'}=\fc 1q+\fc 12$. If $\apa_{i}q\ge2$, that is 
\[
\fc 12-\fc{\apa_{i}}2\le\fc 1r,\tag{Condition 5}
\]
then we get from \eqref{eq:lbest}
\begin{equation}
\||\eta|^{\apa_{i}}|\na\eta|(\tau)\|_{L_{x}^{r'}}\le G_{d,\apa_{i}q,r}^{\apa_{i}}e^{-c_{1}\la\tau}.\label{eq:T2_P2_E4}
\end{equation}
The lower bound in (Condition 5) is less than $1/d$ since $2(\fc 12-\fc 1d)<\apa_{1}$
(this is where we need this requirement). Moreover, $r$ can be $\infty$
if $\apa_{1}\ge1$.

\emph{Step 3. Estimate of $H$.} Suppose 
\begin{equation}
\fc 1{r'}<\fc 2d+\fc 2d(\fc 1{\apa_{1}}-\fc 12),\label{eq:preC6}
\end{equation}
equivalently 
\[
1-\fc 2d\left(\fc 1{\apa_{1}}+\fc 12\right)<\fc 1r.\tag{Condition 6}
\]
One can check that the lower bound is less than $1/d$ by $d\le3$
and $\apa_{1}<2$, and is negative if $d=1$. From \eqref{eq:preC6},
by fixing a small enough $\vn>0$, we have 
\[
\fc 1q:=\fc 2d-\vn\ge 0,
\quad\fc 1p:=\fc 2d(\fc 1{\apa_{1}}-\fc 12)-\vn\ge 0,
\]
and 
\[
\fc 1{s_{2}}:=\fc 1p+\fc 1q
=\fc 2d+\fc 2d(\fc 1{\apa_{1}}-\fc 12)-2\vn>\fc 1{r'}.
\]
From Lemma \ref{lem:imp} (H1), we get
\begin{equation}
\|\na H(\tau)\|_{L_{x}^{r'}}\lesssim(\sum_{i=1,2}A_{\apa_{i}q}^{\apa_{i}}B_{p})^{s_{2}/r'}(\sum_{i=1,2}A_{\infty}^{\apa_{i}}B_{\infty})^{1-s_{2}/r'}e^{-a\min(\apa_{1},1)(1-s_{2}/r')v_{*}\tau}.\label{eq:T2_P2_EH}
\end{equation}
Since $\fc 1q<\fc 2d$ and $\fc 1p<\fc 2d(\fc 1{\apa_{1}}-\fc 12)$,
we have $\apa_{1}q\in\ml C_{A}$ and $p\in\ml C_{B}$. 

From the above discussions, we get the following conclusion: Suppose
$r>d$ is sufficiently close to $d$, $c_{1}$ satisfies \eqref{eq:c1cond},
and $v_{*}$ satisfies 
\begin{equation}
a\min(\apa_{1},1)(1-s_{2}/r')v_{*}\ge c_{1}\la.\label{eq:Thm2_vstar2}
\end{equation}
Then we have 
\begin{align}
\|\na\Phi\eta(t)\|_{L_{x}^{r}} & \lesssim\int_{t}^{\infty}|t-\tau|^{-d(\fc 12-\fc 1r)}E_{2}e^{-c_{1}\la\tau}\, d\tau\nonumber \\
 & =E_{2}\Ga(1-d(\fc 12-\fc 1r))(c_{1}\la)^{d(\fc 12-\fc 1r)-1}e^{-c_{1}\la t},\label{eq:Thm2_gLr}
\end{align}
where $E_2$ is obtained by collecting the coefficients
in \eqref{eq:T2_P2_E1}, \eqref{eq:T2_P2_E2}, 
\eqref{eq:T2_P2_E3}, \eqref{eq:T2_P2_E4}, and \eqref{eq:T2_P2_EH}.

\medskip

\textbf{Part 3. Estimate of $\|\na\Phi\eta(t)\|_{L_{x}^{2}}$.} We have 
\[
\|\na\Phi\eta(t)\|_{L_{x}^{2}}\lesssim\int_{t}^{\infty}(\|\na G(\tau)\|_{L_{x}^{2}}+\|\na H(\tau)\|_{L_{x}^{2}})\, d\tau.
\]
We can imitate Part 2 to obtain all the needed estimates. We summarize
them below.
\begin{enumerate}
\item There is no need of Step 1.
\item For the four sub-steps in Step 2, simply replace ``$r$'' by ``$2$''
(except those of $G_{d,p,r}$ and $\ta$ in using \eqref{eq:lbest}),
we have the following results:

\begin{itemize}
\item[(2-1)]  $\||\eta||T|^{\apa_{i}-1}|\na T|(\tau)\|_{L_{x}^{2}}\lesssim A_{\infty}^{\apa_{i}-1}B_{\infty}e^{-c_{1}\la\tau}$.
\item[(2-2)]  $\||\eta|^{\apa_{i}}|\na T|(\tau)\|_{L_{x}^{2}}\lesssim B_{p}G_{d,\apa_{i}q,r}^{\apa_{i}}e^{-\apa_{i}(1-\ta+c_{1}\ta)\la\tau}$,
where $q=\max(2,2/\apa_{1})$, $p$ is such that $\fc 12=\fc 1q+\fc 1p$,
and $\ta=(\fc 12-\fc 1{\apa_{i}q})(\fc 12+\fc 1d-\fc 1r)^{-1}$. It's
easy to check that $p\in\ml C_{B}$, and $\apa_{i}(1-\ta+c_{1}\ta)\ge c_{1}$
as long as \eqref{eq:c1cond} holds.
\item[(2-3)]  $\||T|^{\apa_{i}}|\na\eta|(\tau)\|_{L_{x}^{2}}\lesssim A_{\infty}^{\apa_{i}}e^{-c_{1}\la\tau}$.
\item[(2-4)]  $\||\eta|^{\apa_{i}}|\na\eta|(\tau)\|_{L_{x}^{2}}\le G_{d,\infty,r}^{\apa_{i}}e^{-c_{1}\la\tau}$. 
\end{itemize}
\item The conclusion of Step 3 is valid with $r$ replaced by $2$. Precisely, we have
\begin{equation}\label{eq:T2_P3_EH}
\|\na H(t)\|_{L_{x}^{2}}\lesssim(\sum_{i=1,2}A_{\apa_{i}q}^{\apa_{i}}B_{p})^{s_{2}/2}(\sum_{i=1,2}A_{\infty}^{\apa_{i}}B_{\infty})^{s_{2}/2}e^{-a\min(\apa_{1},1)(1-s_{2}/2)v_{*}t},
\end{equation}
where $s_{2},p,q$ can be the same as given there.
\end{enumerate}
Thus, if \eqref{eq:c1cond} and \eqref{eq:Thm2_vstar2} hold, we have
\begin{equation}
\|\na\Phi\eta(t)\|_{L_{x}^{2}}\lesssim\int_{t}^{\infty}E_{3}e^{-c_{1}\la\tau}\, d\tau=E_{3}(c_{1}\la)^{-1}e^{-c_{1}\la t},\label{eq:Thm2_gL2}
\end{equation}
where $E_{3}$ is obtained by collecting the coefficients in 
(2-1) -- (2-4) and \eqref{eq:T2_P3_EH}.

The conclusions of the three Parts (namely \eqref{eq:Thm2_L2},
 \eqref{eq:Thm2_gLr} and \eqref{eq:Thm2_gL2}) provide the needed 
estimates for $\Phi$ to be a self-mapping. Similarly we can derive 
the estimates for $\Phi$ to be contractive, and the theorem is true 
by Lemma \ref{lem:compete}.\end{proof}
\begin{rem}
\label{rem:needofg2} Our assertion will be weaker without 
considering the $\|\na\eta(t)\|_{L_{x}^{2}}$ control. 
Precisely, without it, due to the necessary modification
of \emph{Step 2-4} in Part 2, (Condition 5) becomes $\fc 12-\fc{\apa_{i}}4\le\fc 1r$.
Thus we need $4(\fc 12-\fc 1d)<\apa_{1}$ 
(for $\fc 1r<\fc 1d$ to be possible).
Also, since $\apa_{1}<2$, $r=\infty$ is not allowed.
\end{rem}

By Theorem \ref{thm:single1},
Theorem \ref{thm:single2}, and the Gagliardo-Nirenberg's 
inequality \eqref{eq:lbest},
we have proved the following

\begin{cor}\label{cor:single-sum}
Let $d\le 3$, and $f$ satisfy Assumptions \ref{assuF} and \ref{assuT}. 
Assume moreover either of the following conditions:
\begin{itemize}
\item[(i)]
$0<\apa_1\le\apa_2<\infty$ if $d=1$.
\item[(ii)]
$2(\fc 12-\fc 1d)<\apa_{1}<2$ if $d=2,3$.
\end{itemize}
Then for any finite $\rho>0$, there exists a 
constant $\la_{0}>0$ such that the following holds:
For $\la_{0}\le\la<\infty$, there
exist solutions of \eqref{eq:NLSE} of the form 
\eqref{eq:T-def}, with
\begin{equation*}
\sup_{t\ge0}e^{\la t}\|\eta(t)\|_{L_{x}^{2}\cap L_{x}^{\infty}}\le\rho.
\end{equation*}
\end{cor}

\section{Mixed dimensional trains \label{sec:MixedT}}

In this section we consider 
mixed trains. It would be good for the reader
to recall the discussion in Section \ref{sub:mixed}.

First we point out a new problem not mentioned in 
Section \ref{sub:mixed}: 
We can't use only the dispersive inequality \eqref{eq:dispersive} 
to construct mixed trains like we did in the previous section.
To explain the problem, we take the 1D-2D
train $T_{1}+\eta_{1}+T_{2}+\eta$ for example. Corresponding to 
this train, we have 
\[
G=f(T_{1}+\eta_{1}+T_{2}+\eta)-f(T_{1}+\eta_{1}+T_{2}).
\]
Suppose we try to find $\eta$ in a Banach space $X$ whose norm 
assumes the exponential decay of $\|\eta(t)\|_{L_{x}^{p}}$
(with possibly several $p$).
Then we have to estimate $\|\Phi\eta(t)\|_{L_{x}^{p}}$. 
To use the dispersive inequality, we can only consider $p\ge 2$. 
Then we have to estimate $\|G(\tau)\|_{L_{x}^{p'}}$, from which 
we will encounter 
(a) $\||\eta|(|T_{1}|+|\eta_{1}|)^{\apa_{i}}(\tau)\|_{L_{x}^{p'}}$
and (b) $\||\eta|^{\apa_{i}+1}(\tau)\|_{L_{x}^{p'}}$ 
(and also $\||\eta||T_{2}|^{\apa_{i}}(\tau)\|_{L_{x}^{p'}}$,
which is not relevant to the problem). For (a), since the 1D objects
only have $L^{\infty}$ bounds in $x_{2}$, not $L^q_{x_2}$ for 
$q<\infty$, we can only estimate as follows:
\[
\||\eta|(|T_{1}|+|\eta_{1}|)^{\apa_{i}}(\tau)\|_{L_{x}^{p'}}\le\|\eta(\tau)\|_{L_{x}^{p'}}\|(|T_{1}|+|\eta_{1}|)(\tau)\|_{L_{x}^{\infty}}^{\apa_{i}}.
\]
Thus we have to also assume the exponential decay of 
$\|\eta(t)\|_{L_{x}^{p'}}$ for the norm of $X$, and hence have to 
estimate $\|\Phi\eta(t)\|_{L_{x}^{p'}}$.
Again, this can be done only if $p'\ge 2$, and hence we must have
$p=p'=2$. Nevertheless, (b) then requires us to estimate 
$\|\eta(\tau)\|_{L_{x}^{2(\apa_{i}+1)}}$, and the construction fails. 
We remark that adding some $\|\na\eta(t)\|_{L^p_x}$ controls in 
the definition of $X$ also results in similar problems.

Due to the above observation, we shall use the Strichartz estimate
to accomplish our task. In the following section, we recall the basic definitions and facts about the Strichartz space, and then give some more specialized inequalities to be used. 

\subsection{Strichartz space \label{sub:STR}}

Let $\ml A=\ml A^{(d)}$ be the set of all pairs $(q,r)$ satisfying
$\fc 2q+\fc dr=\fc d2$, with $2\le r\le r_{\max}$ or equivalently
$q_{\min}\le q\le\infty$, where
\begin{equation}
r_{\max}=r_{\max}^{(d)}=\left\{ \begin{array}{ll}
\infty & \mbox{if }d=1\\
4 & \mbox{if }d=2\\
\fc{2d}{d-2} & \mbox{if }d\ge3,
\end{array}\right.\quad\mbox{and}\quad q_{\min}=q_{\min}^{(d)}=\left\{ \begin{array}{ll}
4 & \mbox{if }d=1\\
4 & \mbox{if }d=2\\
2 & \mbox{if }d\ge3.
\end{array}\right.\label{eq:rMqm}
\end{equation}
Thus $\ml A$ is the set of all (Schr\"{o}dinger) admissible
pairs if $d\ne2$. For $d=2$, we take $r_{\max}<\infty $ to avoid the forbidden endpoint, and 
$r_{\max}$ can actually be any finite
number no less than $4$ for our approach. We set it to be $4$ for preciseness.

For $\tau\ge0$, we abbreviate $L^{q}([\tau,\infty),L^{r}(\mb R^{d}))$
as $L_{t}^{q}L_{x}^{r}(\tau)$, or even $L_{t}^{q}L_{x}^{r}$ when
the time interval is clear. We'll abuse notation and write
$L_{t}^{q}L_{x}^{r}(t)$, where the two ``$t$'' should not cause
confusion. Define the Strichartz space 
\[
S(t)\coloneqq L_{t}^{\infty}L_{x}^{2}(t)\cap L_{t}^{q_{\min}}L_{x}^{r_{\max}}(t),
\]
with norm
\[
\|\cdot\|_{S(t)}\coloneqq\max(\|\cdot\|_{L_{t}^{\infty}L_{x}^{2}(t)},\|\cdot\|_{L_{t}^{q_{\min}}L_{x}^{r_{\max}}(t)}).
\]
By interpolation, 
\[
S(t)=\bigcap_{(q,r)\in\ml A}L_{t}^{q}L_{x}^{r}(t),\quad\mbox{and}\quad\|\cdot\|_{S(t)}=\sup_{(q,r)\in\ml A}\|\cdot\|_{L_{t}^{q}L_{x}^{r}(t)}.
\]
Denote the dual space of $S(t)$ by $N(t)$. For $(q,r)\in\ml A$,
a function $\xi\in L_{t}^{q'}L_{x}^{r'}(t)$ is regarded as an element
in $N(t)$ by letting
\[
\langle\xi,\eta\rangle_{N(t),S(t)}\coloneqq\int_{t}^{\infty}\int_{\mb R^{d}}\xi(s,x)\eta(s,x)\, dxds\quad(\mbox{for }\eta\in S(t)).
\]
In this way, we have $|\langle\xi,\eta\rangle_{N(t),S(t)}|\le\|\xi\|_{L_{t}^{q'}L_{x}^{r'}(t)}\|\eta\|_{S(t)}$,
and hence 
\[
\|\xi\|_{N(t)}\le\|\xi\|_{L_{t}^{q'}L_{x}^{r'}(t)}.
\]

For $\la>0$ and $t_{0}\ge0$, we define $S_{\la,t_{0}}$ to be the
class of all $\eta\in S(t_{0})$ such that
\[
\|\eta\|_{S_{\la,t_{0}}}\coloneqq\sup_{t\ge t_{0}}e^{\la t}\|\eta\|_{S(t)}<\infty.
\]
By definition, $\|\eta\|_{S(t)}\le\|\eta\|_{S_{\la,t_{0}}}e^{-\la t}$
for $t\ge t_{0}$. In particular, since $\|\eta\|_{L_{t}^{\infty}L_{x}^{2}(t)}\le\|\eta\|_{S(t)}$,
we have 
\begin{equation}
\|\eta(t)\|_{L_{x}^{2}}\le\|\eta\|_{S_{\la,t_{0}}}e^{-\la t}\quad\mbox{for (almost all)}\quad t\ge t_{0}.\label{eq:StrL2}
\end{equation}

In the rest of this section we prove some useful inequalities, 
particularly Lemma \ref{lem:main}. First, we give a fact arising
from a proof step of \cite[Proposition 2.4]{LeCoz_Li_Tsai}. It might
be of independent interest.
\begin{prop}
\label{prop:subadest} Given $0<p\le q\le\infty$ and $\la>0$. If
$u:[0,\infty)\to[0,\infty]$ satisfies $\|u\|_{L^{q}([t,\infty))}\le e^{-\la t}$
for all $t\ge0$, then
\begin{equation}
\|u\|_{L^{p}([t,\infty))}\le\wt C\la^{\fc 1q-\fc 1p}e^{-\la t}\quad\forall\, t\ge0,\label{eq:claim}
\end{equation}
where we can choose $\wt C=\wt C(p)$ in such a way that $\wt C\le(1-e^{-1})^{-1}$
for $p\ge1$. \end{prop}
\begin{proof}
We consider three cases separately. 
\begin{enumerate}
\item If $p=q$, \eqref{eq:claim} is trivially true with $\wt C=1$.
\item If $p<q=\infty$, we have $u(t)\le e^{-\la t}$ for a.a. $t\in[0,\infty)$,
and hence 
\[
\|u\|_{L^{p}([t,\infty))}^{p}\le\int_{t}^{\infty}e^{-p\la\tau}\, d\tau=(p\la)^{-1}e^{-p\la t}.
\]
So \eqref{eq:claim} is true with $\wt C=p^{-1/p}$.
\item Suppose $p<q<\infty$. For fixed $t\ge0$, let $\{t_{k}\}_{k=0}^{\infty}$
be a sequence satisfying $t_{0}=t$ and $t_{k}\nearrow\infty$. Then
\begin{align*}
\|u\|_{L^{p}([t,\infty))}^{p} & =\sum_{k=0}^{\infty}\int_{t_{k}}^{t_{k+1}}u(\tau)^{p}\, d\tau\\
 & \le\sum_{k=0}^{\infty}\Bigl(\int_{t_{k}}^{t_{k+1}}u(\tau)^{q}\, d\tau\Bigr)^{p/q}(t_{k+1}-t_{k})^{1-p/q}\\
 & \le\sum_{k=0}^{\infty}\|u\|_{L^{q}([t_{k},\infty))}^{p}(t_{k+1}-t_{k})^{1-p/q}\\
 & \le\sum_{k=0}^{\infty}e^{-p\la t_{k}}(t_{k+1}-t_{k})^{1-p/q}.
\end{align*}
Letting $t_{k}=t+\fc k{\la}$, we get 
\[
\|u\|_{L^{p}([t,\infty))}^{p}\le(1-e^{-p})^{-1}\la^{p/q-1}e^{-p\la t}.
\]
Thus \eqref{eq:claim} is true with\textit{\textcolor{red}{{} }}$\wt C=(1-e^{-p})^{-1/p}$. 
\end{enumerate}
Comparing the three cases, we see $\wt C\le(1-e^{-1})^{-1}$ for $p\ge1$.
\end{proof}


\begin{defn}\label{def:subad}
If $(q,r)\in\ml A$ and $0<p\le q$, we call $(p,r)$ \textit{sub-admissible.}
Thus $(p,r)$ is sub-admissible if and only if $2\le r\le r_{\max}$,
$p>0$, and $\fc 2p+\fc dr\ge\fc d2.$
\end{defn}

\begin{cor}
\label{cor:subadest} Let $\eta\in S_{\la,t_{0}}$. If $(q,r)\in\ml A$
and $(p,r)$ is sub-admissible, then
\[
\|\eta\|_{L_{t}^{p}L_{x}^{r}(t)}\lesssim\la^{\fc 1q-\fc 1p}\|\eta\|_{S_{\la,t_{0}}}e^{-\la t}=\la^{\fc 12(\fc d2-\fc dr-\fc 2p)}\|\eta\|_{S_{\la,t_{0}}}e^{-\la t}\quad(t\ge t_{0}).
\]
\end{cor}
\begin{proof}
The case of $\eta=0$ is trivial. Assume $\eta\ne0$. Define $u:[0,\infty)\to[0,\infty]$
by 
\[
u(t)=\fc{e^{\la t_{0}}}{\|\eta\|_{S_{\la,t_{0}}}}\|\eta(t+t_{0})\|_{L_{x}^{r}},
\]
then $\|u\|_{L^{q}([t,\infty))}\le e^{-\la t}$ for $t\ge0$. By Proposition
\ref{prop:subadest}, $\|u\|_{L^{p}([t,\infty))}$ satisfies \eqref{eq:claim},
which gives what we want to show. \end{proof}

Definition \ref{def:subad} and Corollary \ref{cor:subadest} are only used in the next result.

\begin{lem}
\label{lem:main}~We have the following estimates.
\begin{itemize}
\item [\textup{(N0)}] Suppose $0\le m\le4/d$. For $u,v\in S_{\la,t_{0}}$,
\[
\||u||v|^{m}\|_{N(t)}\lesssim\la^{-1+dm/4}\|u\|_{S_{\la,t_{0}}}\|v\|_{S_{\la,t_{0}}}^{m}e^{-(m+1)\la t}\quad\forall\, t\ge t_{0}.
\]

\item [\textup{(N1)}] Suppose $0\le m<\apa_{\max}$.
For $u,v\in S_{\la,t_{0}}$ such that $|\na v|\in S_{\la,t_{0}}$,
\[
\||u||v|^{m}\|_{N(t)}\lesssim_{d,m}\la^{-\mu}\|u\|_{S_{\la,t_{0}}}\|v\|_{S_{\la,t_{0}}}^{m(1-b)}\|\na v\|_{S_{\la,t_{0}}}^{mb}e^{-(m+1)\la t}\quad\forall\, t\ge t_{0},
\]
for some $b=b(d,m)\in[0,1]$ and $\mu=\mu(d,m)>0$ (given explicitly
in the proof).
\end{itemize}
\end{lem}
\begin{proof}
Consider (N0). For $(q,r)\in\ml A$, 
\begin{equation}
\||u||v|^{m}\|_{N(t)}\le\||u||v|^{m}\|_{L_{t}^{q'}L_{x}^{r'}(t)}\le\|u\|_{L_{t}^{(m+1)q'}L_{x}^{(m+1)r'}(t)}\|v\|_{L_{t}^{(m+1)q'}L_{x}^{(m+1)r'}(t)}^{m}.\label{eq:firstest}
\end{equation}
We want to show that there is $(q,r)\in\ml A$ such that $((m+1)q',(m+1)r')$
is sub-admissible. That is, there are $q,r$ such that the following
conditions hold ((i),(ii) $\Leftrightarrow(q,r)\in\ml A$; (iii),(iv)
$\Leftrightarrow((m+1)q',(m+1)r')$ is sub-admissible):
\begin{itemize}
\item[(i)]  $\fc 2{q'}+\fc d{r'}=2+\fc d2$.
\item[(ii)]  $r_{\max}'\le r'\le2$.
\item[(iii)]  $\fc 2{(m+1)q'}+\fc d{(m+1)r'}\ge\fc d2$.
\item[(iv)]  $2\le(m+1)r'\le r_{\max}$.
\end{itemize}
It's enough to prove the existence of $r'$ satisfying (ii) and (iv),
since (ii) implies the existence of $q$ such that (i) holds, and
then (iii) also holds by $m\le4/d$. Now (ii) and (iv) are satisfied
by some $r'$ if and only if 
\[
\fc 2{m+1}\le2\quad\mbox{and}\quad r_{\max}'\le\fc{r_{\max}}{m+1},
\]
that is
\begin{equation}
1\le m+1\le\fc{r_{\max}}{r_{\max}'}=\left\{ \begin{array}{ll}
\infty & \mbox{if }d=1\\
3 & \mbox{if }d=2\\
\fc{d+2}{d-2} & \mbox{if }d\ge3.
\end{array}\right.\label{eq:rofm}
\end{equation}
Since $0\le m\le4/d$, \eqref{eq:rofm} is satisfied. (This is where
we need $r_{\max}\ge4$ for $d=2$.) Now let $(q,r)\in\ml A$ be such
that $((m+1)q',(m+1)r')$ is sub-admissible, then \eqref{eq:firstest}
and Corollary \ref{cor:subadest} imply
\[
\||u||v|^{m}\|_{N(t)}\le\wt C^{m+1}\la^{-1+dm/4}\|u\|_{S_{\la,t_{0}}}\|v\|_{S_{\la,t_{0}}}^{m}e^{-(m+1)\la t}.
\]
This proves (N0).

Now we consider (N1). The case of $m=0$ is justified in (N0), so
assume $m>0$. Then, for $(q,r)\in\ml A$ and any $0\le\ta,\phi\le1$,
\begin{equation}
\||u||v|^{m}\|_{N(t)}\le\||u||v|^{m}\|_{L_{t}^{q'}L_{x}^{r'}(t)}\le\|u\|_{L_{t}^{q'/(1-\ta)}L_{x}^{r'/(1-\phi)}(t)}\|v\|_{L_{t}^{mq'/\ta}L_{x}^{mr'/\phi}(t)}^{m}.\label{eq:B1}
\end{equation}
By the Gagliardo-Nirenberg's inequality, if $p,b$ are two numbers
satisfying $p\ge1$, 
\[
0\le b\le1,\ b\ne1\ \mbox{if}\ p=d>1,\tag{b1}
\]
and 
\begin{equation}
\fc{\phi}{mr'}=\fc 1p-\fc bd,\label{eq:GNC}
\end{equation}
then we have 
\begin{align}
\|v\|_{L_{t}^{mq'/\ta}L_{x}^{mr'/\phi}} & =\left\Vert \|v\|_{L_{x}^{mr'/\phi}}\right\Vert _{L_{t}^{mq'/\ta}}\nonumber \\
 & \lesssim_{d,p,b}\left\Vert \|\na v\|_{L_{x}^{p}}^{b}\|v\|_{L_{x}^{p}}^{1-b}\right\Vert _{L_{t}^{mq'/\ta}}\nonumber \\
 & \le\left\Vert \|\na v\|_{L_{x}^{p}}^{b}\right\Vert _{L_{t}^{mq'/(\ta b)}}\left\Vert \|v\|_{L_{x}^{p}}^{1-b}\right\Vert _{L_{t}^{mq'/(\ta(1-b))}}\nonumber \\
 & =\|\na v\|_{L_{t}^{mq'/\ta}L_{x}^{p}}^{b}\|v\|_{L_{t}^{mq'/\ta}L_{x}^{p}}^{1-b}.\label{eq:B1v}
\end{align}
Consider $p\ge2$. Let $\ta=m/(m+1)$ and $\phi=1-r'/p$ (which lies in $[0,1]$
since $p\ge2$), then \eqref{eq:B1} and \eqref{eq:B1v} give 
\begin{equation}
\||u||v|^{m}\|_{N(t)}\lesssim_{d,p,b}\|u\|_{L^{(m+1)q'}L^{p}(t)}\|v\|_{L^{(m+1)q'}L^{p}(t)}^{m(1-b)}\|\na v\|_{L^{(m+1)q'}L^{p}(t)}^{mb},\label{eq:theGineq}
\end{equation}
where 
\begin{equation}
\fc 1{r'}=\fc{m+1}p-\fc{mb}d\label{eq:GNagain}
\end{equation}
from \eqref{eq:GNC}. In summary, for the validity of \eqref{eq:theGineq},
we need (i) $(q,r)\in\ml A$, (ii) $p\ge2$, (iii) (b1) holds, and
(iv) \eqref{eq:GNagain} holds. For the existence of such $q,r,p,b$,
it suffices to show that there exist $p,b$ satisfying (ii), (iii),
and 
\begin{equation}
\fc 12\le\fc{m+1}p-\fc{mb}d\le\fc 1{r_{\max}'}.\label{eq:preb2}
\end{equation}
Since then, by defining $r'$ by \eqref{eq:GNagain} (hence (iv) holds),
there is $q$ such that (i) holds. Notice that \eqref{eq:preb2} is equivalent to
\[
\fc dm\Bigl(\fc{m+1}p-\fc 1{r_{\max}'}\Bigr)\le b\le\fc dm\Bigl(\fc{m+1}p-\fc 12\Bigr).\tag{b2}
\]
Hence we need $p \ge 2$ and $b$ satisfying (b1) and (b2).
Moreover, we want to choose $p,b$ so that $((m+1)q',p)$ is ``strictly'' 
sub-admissible, i.e. $2\le p \le r_{\max}$ and 
\begin{equation}
\mu:=\fc{m+1}2\Bigl(\fc dp+\fc 2{(m+1)q'} -\fc d2\Bigr)=1-\fc m2(\fc d2-b)>0,
\label{eq:4.12}
\end{equation}
where for the equality we use $(q,r)\in\ml A$ and \eqref{eq:GNagain}.
Once we have such $p,b$, then by Corollary \ref{cor:subadest}, 
\eqref{eq:theGineq} gives 
\[
\||u||v|^{m}\|_{N(t)}\lesssim_{d,m}\wt C^{m+1}\la^{-\mu} 
\|u\|_{S_{\la,t_{0}}}\|v\|_{S_{\la,t_{0}}}^{m(1-b)}
\|\na v\|_{S_{\la,t_{0}}}^{mb}e^{-(m+1)\la t},
\]
which is exactly what we want to prove. 

We give possible choices of $p,b$ in the following. 
Notice that $\mu>0$ is trivial if $d\le 2$.

\begin{enumerate}
\item If $d=1$, we can choose $p=2$, and $b$ any number satisfying $\max(\fc{m-1}{2m},0)\le b\le\fc 12$. 
\item If $d=2$, we can choose $p=2$, and $b$ any number satisfying $\max(\fc{2m-1}{2m},0)\le b<1$. 
\item If $d\ge3$, there exists $b$ satisfying (b1) and (b2) if 
\[
\fc{2(m+1)d}{2(m+1)+d}\le p\le 2(m+1),
\]
where the lower bound might be larger than $2$. We consider two cases:

\begin{enumerate}
\item If $0<m\le\fc 2{d-2}$, we can choose $p=2$, and $b$ any number
satisfying $\max(\fc d2-\fc 1m,0)\le b\le1$.
One has $\mu \ge 1-\fc m2(\fc d2-(\fc d2-\fc 1m))=1/2$.

\item If $\fc 2{d-2}<m<\apa_{\max}=\fc 4{d-2}$, we can choose 
$p=\fc{2(m+1)d}{2(m+1)+d}$,
and $b=1$ (the only choice). It is easy to
check that $2\le p \le r_{\max}$ and $\mu>0$.\qedhere
\end{enumerate}
\end{enumerate}

\end{proof}

\subsection{Construction of $e$D-$d$D trains \label{sub:eDdDtrain}}

Consider $1\le e<d$. Let $x=(x',x'')$, where $x'=(x_{1},\ldots,x_{e})$
and $x''=(x_{e+1},\ldots,x_{d})$. In this subsection
we construct mixed dimensional soliton trains of the form
\begin{equation}\label{eq:mixedtrain}
u = T_e + \eta_e + T_d + \eta,
\end{equation}
where 
\begin{align*}
T_e = \sum_{k\in\mb N} R_{e;k}(t,x'),\quad
T_d = \sum_{j\in\mb N} R_{d;j}(t,x),
\end{align*}
with $R_{e;k}$ and $R_{d;j}$ being 
$e$D and $d$D solitons as
given by \eqref{eq:expofsol}, with initial positions assumed 
to be the origin for simplicity.
(The reservation of $j$ for the indices of the $d$D solitons 
and $k$ for those of the $e$D solitons will be convenient.) 
The $e$D error $\eta_e=\eta_e(t,x')$ is such that $T_e+\eta_e$ is 
itself an $e$D train (solution of \eqref{eq:NLSE}), 
whose existence will be provided by the 
previous section. And $\eta=\eta(t,x)$ is the remaining error 
to be found.

Denote the frequencies of
$R_{d;j}$ and $R_{e;k}$ by $\oa_{j}$ and $\sa_{k}$; and 
the velocities by 
\[
v_{j}=(v_{j,1},v_{j,2},\ldots,v_{j,d})\quad\mbox{and}\quad u_{k}=(u_{k,1},\ldots,u_{k,e}).
\]
(Their corresponding bound states and phases will not be used 
explicitly, and hence there is no need to introduce notations 
for them.) $R_{e;k}$ is naturally
regarded as a lower dimensional soliton in $\mb R_{x}^{d}$ by 
considering $R_{e;k}(t,x)\equiv R_{e;k}(t,x')$, with velocity 
$(u_{k,1},\ldots,u_{k,e},0,\ldots,0)$. 

Besides the above, some more modifications of notation given 
in the previous section have to be
made, and some anisotropic generalizations need to be introduced.
We summarize them in the following.
\begin{enumerate}
\item We'll write $A_{d;p}$ for $A_{p}(\{\oa_{j}\})$ and $B_{d;p}$ for
$B_{p}(\{\oa_{j}\},\{v_{j}\})$, as defined in \eqref{eq:defnAB}.
Similarly, we write
\[
\begin{aligned}A_{e;p} & =A_{e;p}(\{\sa_{k}\})=\Bigl(\sum_{k}\sa_{k}^{\min(1,p)(\fc 1{\apa_{1}}-\fc e{2p})}\Bigr)^{\max(1,p^{-1})}\\
B_{e;p} & =B_{e;p}(\{\sa_{k}\},\{u_{k}\})=\Bigl(\sum_{k}\langle u_{k}\rangle^{\min(1,p)}\sa_{k}^{\min(1,p)(\fc 1{\apa_{1}}-\fc e{2p})}\Bigr)^{\max(1,p^{-1})}.
\end{aligned}
\]

\item For $0<p,q\le\infty$, we abbreviate the space $L^{p}(\mb R^{e},L^{q}(\mb R^{d-e}))$
as $L_{x'}^{p}L_{x''}^{q}$. Recall that, for $u:\mb R^{d}\to\mb C$,
\[
\|u\|_{L_{x'}^{p}L_{x''}^{q}}\coloneqq\left\Vert \|u(x',x'')\|_{L_{x''}^{q}}\right\Vert _{L_{x'}^{p}}=\left(\int_{\mb R^{e}}\Bigl(\int_{\mb R^{d-e}}|u(x',x'')|^{q}\, dx''\Bigr)^{p/q}\, dx'\right)^{1/p}.
\]
In particular $L_{x}^{p}=L_{x'}^{p}L_{x''}^{p}$ with exactly the
same norm. The following generalizations are straightforward, hence
we only give them without proof. We have 
\[
\|R_{d;j}\|_{L_{x'}^{p}L_{x''}^{q}}\le D_{p,q}\oa_{j}^{\fc 1{\apa_{1}}-\fc e{2p}-\fc{d-e}{2q}},\quad\mbox{and}\quad\|\na R_{d;j}\|_{L_{x'}^{p}L_{x''}^{q}}\le D_{p,q}\langle v_{j}\rangle\oa_{j}^{\fc 1{\apa_{1}}-\fc e{2p}-\fc{d-e}{2q}},
\]
where $D_{p,q}=D\|e^{-a|y|}\|_{L_{y'}^{p}L_{y''}^{q}}\le D(\fc{2\sqrt{d}}{ap})^{e/p}(\fc{2\sqrt{d}}{aq})^{(d-e)/q}$.
By the same reason as in Remark \ref{rem:p_lb}, we'll absorb $D_{p,q}$
into $\lesssim$. By a similar result of Lemma \ref{lem:moving},
we have 
\[
\|\sum_{j}|R_{d;j}|\|_{L_{x'}^{p}L_{x''}^{q}}\lesssim A_{d;p,q},\quad\mbox{and}\quad\|\sum_{j}|\na R_{d;j}|\|_{L_{x'}^{p}L_{x''}^{q}}\lesssim B_{d;p,q},
\]
where 
\begin{align*}
A_{d;p,q} & \coloneqq\Bigl(\sum_{j}\oa_{j}^{\min(1,p,q)(\fc 1{\apa_{1}}-\fc e{2p}-\fc{d-e}{2q})}\Bigr)^{\max(1,p^{-1},q^{-1})},\\
B_{d;p,q} & \coloneqq\Bigl(\sum_{j}\langle v_{j}\rangle^{\min(1,p,q)}\oa_{j}^{\min(1,p,q)(\fc 1{\apa_{1}}-\fc e{2p}-\fc{d-e}{2q})}\Bigr)^{\max(1,p^{-1},q^{-1})}.
\end{align*}

\item We need all the solitons in both sequences to be separated, hence
we define 
\[
v_{*}=\min(v_{*}(T_{e}),v_{*}(T_{d}),v_{*}(T_{e},T_{d})),
\]
where $v_{*}(T_{e})$ and $v_{*}(T_{d})$ are as defined by \eqref{eq:vstar},
and 
\begin{equation}
v_{*}(T_{e},T_{d})\coloneqq\inf_{j,k\in\mb N}\min(\sa_{k}^{1/2},\oa_{j}^{1/2})|u_{k}-v_{j}'|.\label{eq:vstar_inter}
\end{equation}
Here $v_{j}'=(v_{j,1},\ldots,v_{j,e})$, the first $e$ components
of $v_{j}$. The convention that we add a coefficient $1/2$ in \eqref{eq:vstar}
but not in \eqref{eq:vstar_inter} is only to simplify some expressions.
\item We write $\ml C_{A}^{(d)}$ for the original $\ml C_{A}$, and $\ml C_{A}^{(e)}$
the $e$ dimensional analogue. For the anisotropic case, we define
\[
\ml C_{A}^{(e,d-e)}\coloneqq\Bigl\{(p,q)\in(0,\infty]\times(0,\infty]:\fc 1{\apa_{1}}-\fc e{2p}-\fc{d-e}{2q}>0\Bigr\}.
\]
Similarly, if $\apa_{1}<2$, we have $\ml C_{B}^{(d)}$, $\ml C_{B}^{(e)}$,
and 
\[
\ml C_{B}^{(e,d-e)}\coloneqq\Bigl\{(p,q)\in(0,\infty]\times(0,\infty]:\fc 1{\apa_{1}}-\fc e{2p}-\fc{d-e}{2q}>\fc 12\Bigr\}.
\]
Lemma \ref{lem:compete} can also be generalized. For example, if
$\apa_{1}<2$, $(q_{1},q_{2})\in\ml C_{A}^{(e,d-e)}$, $(p_{1},p_{2})\in\ml C_{B}^{(e,d-e)}$,
then we can choose $\{\oa_{j}\}$ and $\{v_{j}\}$ so that $A_{d;q_{1},q_{2}}$
and $B_{d;p_{1},p_{2}}$ are as small as we like, and $v_{*}$ as
large as we like (see Appendix \ref{appI}). We shall not give a description of all the needed 
facts, but just claim that, as before, it suffices to check that all the indices of
$A,B$ appearing in our proofs lie in their corresponding
controllable class $\ml C$. 
\end{enumerate}

%
To construct solutions of the form \eqref{eq:mixedtrain},
as discussed in Section \ref{sub:mixed}, we consider the operator
$\Phi$ in \eqref{eq:Duhamel} with source term $G+H$, where 
\begin{align*}
G & =f(T_{e}+\eta_{e}+T_{d}+\eta)-f(T_{e}+\eta_{e}+T_{d}),\\
H & =f(T_{e}+\eta_{e}+T_{d})-f(T_{e}+\eta_{e})-\sum_{j}f(R_{d;j}).
\end{align*}
For convenience, we further divide $H$ into $H_{1}+H_{2}$, where 
\begin{align*}
H_{1} & =f(T_{e}+\eta_{e}+T_{d})-f(T_{e}+\eta_{e})-f(T_{d}),\\
H_{2} & =f(T_{d})-\sum_{j}f(R_{d;j}).
\end{align*}
The Strichartz estimate asserts 
\begin{equation}
\|\Phi\eta\|_{S(t)}\lesssim\|G+H_1+H_2\|_{N(t)},\quad\mbox{and}\quad
\|\na\Phi\eta\|_{S(t)}\lesssim\|\na G+\na H_1+\na H_2\|_{N(t)}.\label{eq:StrEst}
\end{equation}
Estimates for $H_{2}$ (or $\na H_{2}$) will be
provided by Lemma \ref{lem:imp}. 

We now give our first main result. Notice that ``$e$'' here corresponds
to the role of ``$d$'' in Section \ref{sec:single}.

\begin{thm}
\label{thm:mixed0} Let $1\le e\le 3$, $e<d\le e+3$,
and $f$ satisfy Assumptions \ref{assuF}, 
\hyperref[assuT]{(T)$_e$}, and \hyperref[assuT]{(T)$_d$}.
Suppose $2(\fc 12-\fc 1e)<\apa_{1}\le\apa_{2}\le 4/d$.
For fixed $0<\rho,t_{0}<\infty$, there is a constant $\la_{0}>0$
such that the following holds: For $\la_{0}\le\la<\infty$, there
exist solutions of \eqref{eq:NLSE} of the form 
\eqref{eq:mixedtrain}, with
\[
\sup_{t\ge t_{0}}e^{\la t}\left\{ \|\eta_{e}(t)\|_{L_{x'}^{2}\cap L_{x'}^{\infty}}+\|\eta\|_{S(t)}\right\} \le\rho.
\]
\end{thm}

\begin{rem*}
We need $\alpha_2 \le 4/d$ so that we can bound 
$\||\eta|^{\alpha_2+1}\|_{N(t)}$ by $\|\eta\|_{S(t)}^{\alpha_2+1}$
from Lemma \ref{lem:main}. We need $d\le e+3$ in estimating
$\||T_e+\eta_e||T_d|^{\apa_1}\|_{N(t)}$.
We may take $t_0=0$ if $ \alpha_2 < 4/d$, or if $\rho$ is sufficiently small.
\end{rem*}

\begin{rem*} 
It's most natural to view the $e$D-$d$D trains as solutions of 
\eqref{eq:NLSE} in $\mb{R}^d_x$, with $d$D solitons being 
``points'' and $e$D solitons lower dimensional objects.
Nevertheless, as we have mentioned in the introduction,
we can also freely regard them as living in an even higher 
dimension, so that both $e,d$ have nonzero
codimensions to the ambient space.
\end{rem*}

\begin{proof}
We will only consider $\rho=2$. The cases of other $\rho$
can be treated similarly.

First, from the assumption, if $e=2,3$, then $d\ge3$, and hence 
$\apa_{1}<2$. Thus, for $e=1,2,3$, if $\la$ is large enough, 
Corollary \ref{cor:single-sum} implies the existence 
of an $e$D train $T_{e}+\eta_{e}$ satisfying 
\begin{equation}
\|\eta_{e}(t)\|_{L_{x'}^{2}\cap L_{x'}^{\infty}}
\le e^{-\la t},\quad\forall t\ge t_{0}.\label{eq:Lerror}
\end{equation}
It remains to prove that $\Phi$ can be a contraction mapping 
on the closed unit ball of $S_{\la,t_{0}}$.
As before, we'll only give estimates for $\Phi$ to be a 
self-mapping.

Suppose $\eta\in S_{\la,t_{0}}$ with $\|\eta\|_{S_{\la,t_{0}}}\le1$,
i.e. $\|\eta\|_{S(t)}\le e^{-\la t}$ for $t\ge t_{0}$. To estimate
$\|\Phi\eta\|_{S_{\la,t_{0}}}$ from the Strichartz estimate \eqref{eq:StrEst},
we have to estimate $\|G\|_{N(t)}$, $\|H_{1}\|_{N(t)}$ and $\|H_{2}\|_{N(t)}$.
Since $\|\cdot\|_{N(t)}\le\|\cdot\|_{L_{t}^{1}L_{x}^{2}}$, we'll
frequently just estimate $\|\cdot\|_{L_{t}^{1}L_{x}^{2}}$. Also repeatedly
used is the fact $\|\eta\|_{L_{t}^{1}L_{x}^{2}(t)}\le\la^{-1}e^{-\la t}$,
obtained from \eqref{eq:StrL2} (or Corollary \ref{cor:subadest}).

\medskip

\textbf{Part 1. Estimate of $\|G\|_{N(t)}$.} We have 
\[
|G|\lesssim\sum_{i=1,2}(|\eta||T_{e}+\eta_{e}+T_{d}|^{\apa_{i}}+|\eta|^{\apa_{i}+1}).
\]
For the first term, we have
\begin{align}
\||\eta||T_{e}+\eta_{e}+T_{d}|^{\apa_{i}}\|_{L_{t}^{1}L_{x}^{2}(t)} & \lesssim\|\eta\|_{L_{t}^{1}L_{x}^{2}(t)}\|T_{e}+\eta_{e}+T_{d}\|_{L_{t}^{\infty}L_{x}^{\infty}(t)}^{\apa_{i}}\nonumber \\
 & \lesssim(A_{e;\infty}+e^{-\la t_{0}}+A_{d;\infty})^{\apa_{i}}\la^{-1}e^{-\la t},\label{eq:T3_G1}
\end{align}
by \eqref{eq:Lerror}. For the second term, since $\apa_{2}\le 4/d$ 
(and hence $\apa_{1}\le 4/d$),
Lemma \ref{lem:main} (N0) implies
\begin{equation}
\||\eta|^{\apa_{i}+1}\|_{N(t)}=\||\eta||\eta|^{\apa_{i}}\|_{N(t)}\lesssim(\la^{-1+d\apa_{i}/4}e^{-\apa_{i}\la t_{0}})e^{-\la t}.\label{eq:T3_G2}
\end{equation}
Notice that for the endpoint case $\apa_{i}=4/d$, the smallness of
the coefficient (obtained by letting $\la$ large) have to be provided
by $e^{-\apa_{i}\la t_{0}}$. This is the reason we consider an initial
time $t_{0}>0$. 
By \eqref{eq:T3_G1} and \eqref{eq:T3_G2} we get
the needed estimate of $\|G\|_{N(t)}$.

\medskip

\textbf{Part 2. Estimate of $\|H_{1}\|_{N(t)}$. }By Corollary \ref{cor:decomp},
\begin{align*}
|H_{1}| & \lesssim\sum_{i=1,2}(|T_{e}+\eta_{e}|^{\max(1,\apa_{i})}|T_{d}|^{\min(1,\apa_{i})}+|T_{e}+\eta_{e}||T_{d}|^{\apa_{i}})
\\
 & =|T_{e}+\eta_{e}||T_{d}|^{\min(1,\apa_{1})}h_{1},
\end{align*}
where 
\[
\|h_{1}\|_{L_{t}^{\infty}L_{x}^{\infty}(t)}\lesssim\sum_{i=1,2}A_{d;\infty}^{\min(1,\apa_{i})-\min(1,\apa_{1})}(A_{e;\infty}+e^{-\la t_{0}}+A_{d;\infty})^{\max(0,\apa_{i}-1)}.
\]
Thus it suffices to estimate $\||\eta_{e}||T_{d}|^{\min(1,\apa_{1})}\|_{N(t)}$
and $\||T_{e}||T_{d}|^{\min(1,\apa_{1})}\|_{N(t)}$. In the following
we denote $\ga=\min(1,\apa_{1})$ to save notation. 

\emph{Part 2-1. Estimate of $\||\eta_{e}||T_{d}|^{\ga}\|_{N(t)}$.
}Since $L_{x}^{2}=L_{x'}^{2}L_{x''}^{2}$,
\begin{align}
\||\eta_{e}||T_{d}|^{\ga}\|_{N(t)} & \le\||\eta_{e}||T_{d}|^{\ga}\|_{L_{t}^{1}L_{x'}^{2}L_{x''}^{2}(t)}\nonumber \\
 & \le\|\eta_{e}\|_{L_{t}^{1}L_{x'}^{2}L_{x''}^{\infty}(t)}\||T_{d}|^{\ga}\|_{L_{t}^{\infty}L_{x'}^{\infty}L_{x''}^{2}(t)}\nonumber \\
 & \lesssim A_{d;\infty,2\ga}^{\ga}\la^{-1}e^{-\la t}.\label{eq:forrmk}
\end{align}
Now $(\infty,2\ga)\in\ml C_{A}^{(e,d-e)}$ means 
$\fc 1{\apa_{1}}>\fc{d-e}{4\ga}$.
If $\ga=\apa_{1}$, it's true since $d-e<4$.
If $\ga=1$, it's true since $\apa_{1}\le4/d$.

\emph{Part 2-2. Estimate of $\||T_{e}||T_{d}|^{\ga}\|_{N(t)}$. }We
first prove the exponential decay of its $L_{x}^{2}$ norm by interpolation.

\emph{Step 1. }For $s\in(0,\infty]$ and $\ta\in[0,1]$,
\begin{equation}
\||T_{e}||T_{d}|^{\ga}\|_{L_{x}^{s}}\le\|T_{e}\|_{L_{x'}^{s/\ta}L_{x''}^{\infty}}\||T_{d}|^{\ga}\|_{L_{x'}^{s/(1-\ta)}L_{x''}^{s}}\lesssim A_{e;s/\ta}A_{d;\ga s/(1-\ta),\ga s}^{\ga}.\label{eq:T3_H1_2-1-1}
\end{equation}
We need $s/\ta\in\ml C_{A}^{(e)}$ and $(\ga s/(1-\ta),\ga s)\in\ml C_{A}^{(e,d-e)}$,
that is
\[
\fc 1{\apa_{1}}>\fc e{2(s/\ta)}\quad\mbox{and}\quad\fc 1{\apa_{1}}>\fc e{2(\ga s/(1-\ta))}+\fc{d-e}{2\ga s},
\]
or equivalently 
\[
s>\max\left(\fc{e\ta\apa_{1}}2,\fc{(d-e\ta)\apa_{1}}{2\ga}\right).
\]
We hope this can be satisfied by some $s<2$, by choosing a suitable
$\ta$. A little computation shows that the minimum of the ``max''
is achieved by letting 
\begin{equation}
\ta=\min\left(\fc d{e(1+\ga)},1\right).\label{eq:Thm3_vta}
\end{equation}
Precisely we have the following alternatives:
\begin{enumerate}
\item If $\ta=\fc d{e(1+\ga)}\le 1$, we get $s>\fc{d\apa_{1}}{2(1+\ga)}$. 
\item If $\ta=1<\fc d{e(1+\ga)}$, we get 
$s>\fc{(d-e)\apa_{1}}{2\ga}$.
\end{enumerate}
It's straightforward to check that, for all $(e,d,\apa_1)$ 
satisfying our assumptions, the above lower bound of $s$
is less than $2$. (Here we use $d-e<4$ again.)
Thus, if $\ta$ is given by \eqref{eq:Thm3_vta},
there exists $0<s_{1}<2$ such that \eqref{eq:T3_H1_2-1-1} holds
with $s=s_{1}$.

\textit{Step 2.} We have 
\begin{align}
\||T_{e}||T_{d}|^{\ga}(\tau)\|_{L_{x}^{\infty}} & \le\|(\sum_{k}|R_{e;k}|)(\sum_{j}|R_{d;j}|)^{\ga}(\tau)\|_{L_{x}^{\infty}}\nonumber \\
 & \le\|(\sum_{k}|R_{e;k}|)(\sum_{j}|R_{d;j}|^{\ga})(\tau)\|_{L_{x}^{\infty}}\quad(\mbox{since }\ga\le1)\nonumber \\
 & =\|\sum_{k,j}|R_{e;k}||R_{d;j}|^{\ga}(\tau)\|_{L_{x}^{\infty}}\nonumber \\
 & \le\sum_{k,j}\||R_{e;k}||R_{d;j}|^{\ga}(\tau)\|_{L_{x}^{\infty}}.\label{eq:T3_H1_2-2-1}
\end{align}
We also have 
\begin{align*}
|R_{e;k}||R_{d;j}|^{\ga}(\tau) & \lesssim\sa_{k}^{\fc 1{\apa_{1}}}\oa_{j}^{\fc{\ga}{\apa_{1}}}e^{-a\sa_{k}^{1/2}|x'-u_{k}\tau|-a\ga\oa_{j}^{1/2}|x-v_{j}\tau|}\\
 & \le\sa_{k}^{\fc{1}{\apa_{1}}}\oa_{j}^{\fc{\ga}{\apa_{1}}}e^{-a\sa_{k}^{1/2}|x'-u_{k}\tau|-a\ga\oa_{j}^{1/2}|x'-v_{j}'\tau|},
\end{align*}
where recall that $v_{j}'=(v_{j,1},\ldots,v_{j,e})$ consists of the first
$e$ components of $v_{j}$. Note that for any $c_{1},c_{2}>0$ and
$w_{1},w_{2}\in\mb R^{n}$, 
\[
b_{1}|x-w_{1}|+b_{2}|x-w_{2}|\ge\min(b_{1},b_{2})(|x-w_{1}|+|x-w_{2}|)\ge\min(b_{1},b_{2})|w_{1}-w_{2}|.
\]
Thus 
\begin{align*}
|R_{e;k}||R_{d;j}|^{\ga}(\tau) & \lesssim\sa_{k}^{\fc 1{\apa_{1}}}\oa_{j}^{\fc{\ga}{\apa_{1}}}e^{-a\ga\min(\sa_{k}^{1/2},\oa_{j}^{1/2})|u_{k}-v_{j}'|\tau}\\
 & \le\sa_{k}^{\fc{1}{\apa_{1}}}\oa_{j}^{\fc{\ga}{\apa_{1}}}e^{-a\min(1,\apa_{1})v_{*}\tau}.
\end{align*}
Taking this into \eqref{eq:T3_H1_2-2-1}, we get
\begin{align}
\||T_{e}||T_{d}|^{\ga}(\tau)\|_{L_{x}^{\infty}} & \lesssim\sum_{k,j}\sa_{k}^{\fc 1{\apa_{1}}}\oa_{j}^{\fc{\ga}{\apa_{1}}}e^{-a\min(1,\apa_{1})v_{*}\tau}\nonumber \\
 & =(\sum_{k}\sa_{k}^{\fc 1{\apa_{1}}})(\sum_{j}\oa_{j}^{\fc{\ga}{\apa_{1}}})e^{-a\min(1,\apa_{1})v_{*}\tau}.\label{eq:T3_H1_2-3-1}
\end{align}
The number $\sum_{j}\oa_{j}^{\fc{\ga}{\apa_{1}}}$ can be controlled
as $A_{d;p}$ as described in Lemma \ref{lem:compete}. For preciseness,
we can fix any $p_{1}\in\ml C_{A}^{(d)}$ close to $d\apa_1/2$ such that 
\[
\fc{\ga}{\apa_{1}}\ge 
\min(1,p_{1})\Bigl(\fc 1{\apa_{1}}-\fc d{2p_{1}}\Bigr),
\]
which implies 
\[
\sum_{j}\oa_{j}^{\fc{\ga}{\apa_{1}}}\lesssim\sum_{j}\oa_{j}^{\min(1,p_{1})(\fc 1{\apa_{1}}-\fc d{2p_{1}})}=A_{d;p_{1}}^{\min(1,p_{1})}.
\]
Thus \eqref{eq:T3_H1_2-3-1} gives
\begin{equation}
\||T_{e}||T_{d}|^{\ga}(\tau)\|_{L_{x}^{\infty}}\lesssim A_{e;\infty}A_{d;p_{1}}^{\min(1,p_{1})}e^{-a\min(1,\apa_{1})v_{*}\tau}.\label{eq:T3_H1_2-4-1}
\end{equation}

From \eqref{eq:T3_H1_2-1-1} (with $\ta$ given by \eqref{eq:Thm3_vta}
and $s=s_{1}<2$) and \eqref{eq:T3_H1_2-4-1}, we get 
\[
\||T_{e}||T_{d}|^{\ga}(\tau)\|_{L_{x}^{2}}\lesssim E_{2}e^{-(1-s_{1}/2)a\min(1,\apa_{1})v_{*}\tau}.
\]
We omit the expression of $E_{2}$, which is obvious while cumbersome.
Suppose 
\begin{equation}
(1-s_{1}/2)a\min(1,\apa_{1})v_{*}\ge\la,
\end{equation}
we get 
\[
\||T_{e}||T_{d}|^{\ga}\|_{N(t)}\le\||T_{e}||T_{d}|^{\ga}\|_{L_{t}^{1}L_{x}^{2}(t)}\lesssim E_{2}\la^{-1}e^{-\la t}.
\]

\textbf{Part 3. Estimate of $\|H_{2}\|_{N(t)}$.} 
Choose $s_2\in (\fc{d\apa_{1}}{2(\apa_{1}+1)},2)$
(it's easy to check that the interval is nonempty). 
Then Lemma \ref{lem:imp} (H0)
implies 
\[
\|H_{2}(\tau)\|_{L_{x}^{2}}\lesssim(\sum_{i=1,2}A_{d;(\apa_{i}+1)s_{2}}^{\apa_{i}+1})^{s_{2}/2}(\sum_{i=1,2}A_{d;\infty}^{\apa_{i}+1})^{1-s_{2}/2}e^{-a(1-s_{2}/2)v_{*}\tau},
\]
with $(\apa_{1}+1)s_{2}\in\ml C_{A}^{(d)}$. Thus, suppose 
\begin{equation}
a(1-s_{2}/2)v_{*}\ge\la,
\end{equation}
we get 
\[
\|H_{2}\|_{N(t)}\le\|H_{2}\|_{L_{t}^{1}L_{x}^{2}(t)}\lesssim(\sum_{i=1,2}A_{d;(\apa_{i}+1)s_{2}}^{\apa_{i}+1})^{s_{2}/2}(\sum_{i=1,2}A_{d;\infty}^{\apa_{i}+1})^{1-s_{2}/2}\la^{-1}e^{-\la t}.
\]

From the conclusions in Part 1, Part 2, and Part 3, 
we are done.\end{proof}
\begin{rem}
Without using the anisotropic estimates for $T_{d}$, our assertions
will be much weaker. For example, consider \eqref{eq:forrmk} in 
\emph{Part 2-1}. If we do not use an anisotropic estimate of $T_{d}$, 
we can only estimate as follows: For any $(q,r)\in\ml A$ 
\begin{align*}
\||\eta_{e}||T_{d}|^{\ga}\|_{N(t)} & \le\||\eta_{e}||T_{d}|^{\ga}\|_{L_{t}^{q'}L_{x}^{r'}(t)}\\
 & \le\|\eta_{e}\|_{L_{t}^{q'}L_{x}^{\infty}(t)}\||T_{d}|^{\ga}\|_{L_{t}^{\infty}L_{x}^{r'}(t)}\lesssim A_{d;\ga r'}^{\ga}\|\eta_{e}\|_{L_{t}^{q'}L_{x}^{\infty}(t)}.
\end{align*}
Now for $\ga r'\in\ml C_{A}^{(d)}$, we need
\begin{align}\label{eq:no_ani}
\fc{d\apa_{1}}{2\ga}<r'\le 2.
\end{align}
If $d\ge 4$, we have $\ga=\min(1,\apa_1)=\apa_1$ (since $\apa_i\le 4/d$),
and \eqref{eq:no_ani} is impossible.
Thus only $d\le 3$ is allowed. Moreover, even for $d\le 3$, if 
$\ga=1$, the endpoint case $\apa_{1}=\apa_{2}=4/d$ is excluded. 
\end{rem}
When $1\le\apa_{1}<2$, Theorem \ref{thm:single2} implies the existence
of an 1D train $T_{1}+\eta_{1}$ such that $\|\eta_{1}(t)\|_{W_{x}^{1,\infty}}$
has exponential decay. This allows us to use the gradient estimate
when $e=1$. Precisely, we can try to construct a mixed train of the
form $T_{1}+\eta_{1}+T_{d}+\eta$ ($d>1$), by assuming the exponential
decay of $\|\na\eta\|_{S(t)}$ (besides $\|\eta\|_{S(t)}$). It turns
out that we can do it only for $d=2$, and under a further restriction
on $\apa_{1}$. The result is not only of its own interest, but also
makes it possible to realize the 1D-2D-3D trains in the next section. 
\begin{thm}
\label{thm:mixed1} Let $e=1$, $d=2$, and $f$ satisfy 
Assumptions \ref{assuF}, \hyperref[assuT]{(T)$_1$}, and
\hyperref[assuT]{(T)$_2$}.
Suppose moreover $1\le\apa_{1}<4/3$. 
Then for any finite $\rho>0$, there is a constant $\la_{0}>0$ such 
that the following holds: For $\la_{0}\le\la<\infty$, there exist 
solutions of \eqref{eq:NLSE} 
of the form \eqref{eq:mixedtrain} 
(namely $T_{1}+\eta_{1}+T_{2}+\eta$) such that 
\[
\sup_{t\ge 0}e^{\la t}\left\{ 
\|\eta_{1}(t)\|_{H_{x_{1}}^{1}\cap W_{x_{1}}^{1,\infty}}
+\|\eta\|_{S(t)}+\|\na\eta\|_{S(t)}
\right\} \le\rho.
\]
\end{thm}

\begin{rem*}
We need $\alpha_1 < 4/3$ and $d=2$ to bound
$\|(T_1+\eta_1)\nabla T_2\|_{N(t)}$.
Note $t\ge t_0=0$ even for large $\rho$.
\end{rem*}

\begin{proof} We will assume $\rho=2$ for simplicity.
For $\la$ no less than some positive number, Theorem \ref{thm:single2}
implies the existence of an 1D train $T_{1}+\eta_{1}$ satisfying
\[
\|\eta_{1}(t)\|_{H_{x_{1}}^{1}\cap W_{x_{1}}^{1,\infty}}\le e^{-\la t},\quad\forall t\ge0.
\]
In the following, we denote $S_{\la,0}$ (i.e. the initial time $t_{0}=0$)
by $S_{\la}$, and let $X$ be the Banach space of all 
$\eta:[0,\infty)\times\mb R^{2}\to\mb C$
such that 
\[
\|\eta\|_{X}\coloneqq\|\eta\|_{S_{\la}}+\|\na\eta\|_{S_{\la}}<\infty.
\]
We'll give estimates for $\Phi$ to be a self-mapping on the 
closed unit ball of $X$.

Suppose $\eta\in X$ with $\|\eta\|_{X}\le1$. 
The estimate of $\|\Phi\eta\|_{S_{\la}}$
is the same as in the proof of Theorem \ref{thm:mixed0}, except for
$\||\eta|^{\apa_{2}+1}\|_{N(t)}$. Since the value of $\apa_{2}$
is not restricted, we use Lemma \ref{lem:main} (N1) instead of (N0)
to obtain 
\[
\||\eta|^{\apa_{2}+1}\|_{N(t)}\lesssim\la^{-\mu}e^{-(\apa_{2}+1)\la t},
\]
for some $\mu=\mu(d=2,\apa_{2})>0$. (For $\||\eta|^{\apa_{1}+1}\|_{N(t)}$,
both (N0) and (N1) work.) We remark that here and later we use $\mu$
as a generic constant, whose value may be different in different places.

Now we estimate $\|\na\Phi\eta\|_{S(t)}$. From \eqref{eq:StrEst},
we have to estimate $\|\na G\|_{N(t)}$, $\|\na H_{1}\|_{N(t)}$,
and $\|\na H_{2}\|_{N(t)}$. 

\medskip

\textbf{Part 1. Estimate of $\|\na G\|_{N(t)}$.} Let $W=T_{1}+\eta_{1}+T_{2}$,
then $G=f(W+\eta)-f(W)$. Since $\apa_{1}\ge1$, \eqref{eq:fineq1'}
implies
\[
|\na G|\lesssim\sum_{i=1,2}\Bigl\{|\eta|(|W|+|\eta|)^{\apa_{i}-1}|\na W|+(|W|+|\eta|)^{\apa_{i}}|\na\eta|\Bigr\}.
\]
Thus we have to estimate the $N(t)$ norm of (1) $|\eta||W|^{\apa_{i}-1}|\na W|$,
(2) $|\eta|^{\apa_{i}}|\na W|$, (3) $|W|^{\apa_{i}}|\na\eta|$, and
(4) $|\eta|^{\apa_{i}}|\na\eta|$. We discuss them in the following. 

\emph{Estimate (1).} We have
\begin{align*}
\||\eta||W|^{\apa_{i}-1}|\na W|\|_{L_{t}^{1}L_{x}^{2}(t)} & \le\|W\|_{L_{t}^{\infty}L_{x}^{\infty}}^{\apa_{i}-1}\|\na W\|_{L_{t}^{\infty}L_{x}^{\infty}}\|\eta\|_{L_{t}^{1}L_{x}^{2}}\\
 & \lesssim(A_{1;\infty}+1+A_{2;\infty})^{\apa_{i}-1}(B_{1;\infty}+1+B_{2;\infty})\la^{-1}e^{-\la t}.
\end{align*}

\emph{Estimate (2).} By Lemma \ref{lem:main} (N1), 
\begin{align*}
\||\eta|^{\apa_{i}}|\na W|\|_{N(t)} & \le\|\na W\|_{L_{t}^{\infty}L_{x}^{\infty}}\||\eta|^{\apa_{i}}\|_{N(t)}\\
 & \lesssim(B_{1;\infty}+1+B_{2;\infty})\la^{-\mu}e^{-\apa_{i}\la t},
\end{align*}
for some $\mu>0$.

\emph{Estimate (3).} 
\[
\||W|^{\apa_{i}}|\na\eta|\|_{L_{t}^{1}L_{x}^{2}(t)}\le\|W\|_{L_{t}^{\infty}L_{x}^{\infty}}^{\apa_{i}}\|\na\eta\|_{L_{t}^{1}L_{x}^{2}(t)}\lesssim(A_{1;\infty}+1+A_{2;\infty})^{\apa_{i}}\la^{-1}e^{-\la t}.
\]

\emph{Estimate (4).} Also by Lemma \ref{lem:main} (N1) (with $u=|\na\eta|$
and $v=\eta$ there), we get 
\[
\||\eta|^{\apa_{i}}|\na\eta|\|_{N(t)}\lesssim\la^{-\mu}e^{-(\apa_{i}+1)\la t},
\]
for some $\mu>0$.

\medskip

\textbf{Part 2. Estimate of $\|\na H_{1}\|_{N(t)}$.} Let $w=T_{1}+\eta_{1}$.
Since $\apa_{1}\ge1$, \eqref{eq:fineq1''} implies
\begin{align*}
|\na H_{1}| & =|\na[f(w+T_{2})-f(w)-f(T_{2})]|\\
 & \lesssim\sum_{i=1,2}(|w|+|T_{2}|)^{\apa_{i}-1}(|w||\na T_{2}|+|T_{2}||\na w|).
\end{align*}
Since 
\begin{equation}
\|(|w|+|T_{2}|)^{\apa_{i}-1}\|_{L_{t}^{\infty}L_{x}^{\infty}(t)}\lesssim(A_{1;\infty}+1+A_{2;\infty})^{\apa_{i}-1},\label{eq:rough}
\end{equation}
it suffices to estimate (1) $|\eta_{1}||\na T_{2}|$, (2) $|T_{2}||\na\eta_{1}|$,
(3) $|T_{1}||\na T_{2}|$, and (4) $|T_{2}||\na T_{1}|$. We discuss
them in the following.

\emph{Estimate (1). }
\begin{align*}
\||\eta_{1}||\na T_{2}|\|_{L_{t}^{1}L_{x}^{2}(t)} & \le\|\eta_{1}\|_{L_{t}^{1}L_{x_{1}}^{2}L_{x_{2}}^{\infty}(t)}\|\na T_{2}\|_{L_{t}^{\infty}L_{x_{1}}^{\infty}L_{x_{2}}^{2}(t)}\\
 & \lesssim B_{2;\infty,2}\la^{-1}e^{-\la t}.
\end{align*}
We need $(\infty,2)\in\ml C_{B}^{(1,1)}$, i.e. $\fc 1{\apa_{1}}-\fc 1{2\cdot\infty}-\fc 1{2\cdot2}>\fc 12$.
This is true by $\apa_{1}<4/3$. 

Notice that if we do not use an
anisotropic estimate for $\na T_{2}$, the requirement becomes 
$\apa_{1}<1$, and the construction fails since we assume 
$\apa_{1}\ge1$. Moreover, it is also due to this part that the 
construction is valid only for $d=2$. Indeed, suppose $d\ge 3$,
with coordinates $x=(x_1,x'')$. If for some admissible $(a',b')$,
\[
\||\eta_{1}||\na T_{2}|\|_{L_{t}^{a}L_{x}^{b}(t)}  \le\|\eta_{1}\|_{L_{t}^{a}L_{x_{1}}^{2}L_{x''}^{\infty}(t)}\|\na T_{2}\|_{L_{t}^{\infty}L_{x_{1}}^{p}L_{x''}^{b}(t)},
\]
where $1/p =1/b-1/2$ and  $(p,b)\in\ml C_{B}^{(1,d-1)}$. It follows
\[
\frac 1{\alpha_1} > \frac 12 + \frac 1{2p}+\frac {d-1}{2b} > \frac 12+ 0+\frac 1b \ge 1,
\]
contradicting $1\le\alpha_1$.

\emph{Estimate (2). }
\begin{align*}
\||T_{2}||\na\eta_{1}|\|_{L_{t}^{1}L_{x}^{2}(t)} & \le\|T_{2}\|_{L_{t}^{\infty}L_{x_{1}}^{\infty}L_{x_{2}}^{2}(t)}\|\na\eta_{1}\|_{L_{t}^{1}L_{x_{1}}^{2}L_{x_{2}}^{\infty}(t)}\\
 & \lesssim A_{2;\infty,2}\la^{-1}e^{-\la t},
\end{align*}
where $(\infty,2)\in\ml C_{A}^{(1,1)}$.

\emph{Estimate (3).} We will prove the exponential decay of  $\|T_{1}\na T_{2}\|_{L_{x}^{2}}$
by interpolation. First, for $s\in(0,\infty]$ and $\ta\in[0,1]$,
\begin{equation}
\||T_{1}||\na T_{2}|\|_{L_{x}^{s}}\le\|T_{1}\|_{L_{x_{1}}^{s/\ta}L_{x_{2}}^{\infty}}\|\na T_{2}\|_{L_{x_{1}}^{s/(1-\ta)}L_{x_{2}}^{s}}\lesssim A_{1;s/\ta}B_{2;s/(1-\ta),s}.\label{eq:Thm4_pS1}
\end{equation}
We need $s/\ta\in\ml C_{A}^{(1)}$ and $(s/(1-\ta),s)\in\ml C_{B}^{(1,1)}$,
that is
\[
\fc 1{\apa_{1}}-\fc 1{2(s/\ta)}>0,\quad\mbox{and}\quad\fc 1{\apa_{1}}-\fc 1{2s/(1-\ta)}-\fc 1{2s}>\fc 12,
\]
or equivalently
\[
s>\max\left(\fc{\ta\apa_{1}}2,\fc{(2-\ta)\apa_{1}}{2-\apa_{1}}\right).
\]
The ``max'' is minimized by letting $\ta=1$, which gives $s>\fc{\apa_{1}}{2-\apa_{1}}$.
Since $\apa_{1}<4/3$, the lower bound is less than $2$. Thus \eqref{eq:Thm4_pS1}
implies 
\begin{equation}
\||T_{1}||\na T_{2}|\|_{L_{x}^{s_{1}}}\lesssim A_{1;s_{1}}B_{2;\infty,s_{1}}\label{eq:Thm4_S1}
\end{equation}
for some $s_{1}<2$. Then consider the supremum estimate. 
\begin{align*}
\||T_{1}||\na T_{2}|(\tau)\|_{L_{x}^{\infty}} & \le\|(\sum_{k}|R_{1;k}|)(\sum_{j}|\na R_{2;j}|)(\tau)\|_{L_{x}^{\infty}}\\
 & \le\sum_{k,j}\||R_{1;k}||\na R_{2;j}|(\tau)\|_{L_{x}^{\infty}}.
\end{align*}
We have
\begin{align*}
|R_{1;k}||\na R_{2;j}|(\tau) & \lesssim\sa_{k}^{\fc 1{\apa_{1}}}\oa_{j}^{\fc 1{\apa_{1}}}\langle v_{j}\rangle e^{-a\sa_{k}^{1/2}|x_{1}-u_{k}\tau|-a\oa_{j}^{1/2}|x-v_{j}\tau|}\\
 & \le\sa_{k}^{\fc 1{\apa_{1}}}\oa_{j}^{\fc 1{\apa_{1}}}\langle v_{j}\rangle e^{-a\sa_{k}^{1/2}|x_{1}-u_{k}\tau|-a\oa_{j}^{1/2}|x_{1}-v_{j,1}\tau|}\\
 & \le\sa_{k}^{\fc 1{\apa_{1}}}\oa_{j}^{\fc 1{\apa_{1}}}\langle v_{j}\rangle e^{-a\min(\sa_{k}^{1/2},\oa_{j}^{1/2})|u_{k}-v_{j,1}|\tau}\\
 & \le\sa_{k}^{\fc 1{\apa_{1}}}\oa_{j}^{\fc 1{\apa_{1}}}\langle v_{j}\rangle e^{-av_{*}\tau}.
\end{align*}
Thus
\begin{align}
\||T_{1}||\na T_{2}|(\tau)\|_{L_{x}^{\infty}} & \lesssim\sum_{k,j}\sa_{k}^{\fc 1{\apa_{1}}}\oa_{j}^{\fc 1{\apa_{1}}}\langle v_{j}\rangle e^{-av_{*}\tau}\nonumber \\
 & =(\sum_{k}\sa_{k}^{\fc 1{\apa_{1}}})(\sum_{j}\oa_{j}^{\fc 1{\apa_{1}}}\langle v_{j}\rangle)e^{-av_{*}\tau}\nonumber \\
 & =A_{1;\infty}B_{2;\infty}e^{-av_{*}\tau}.\label{eq:Thm4_infty}
\end{align}
From \eqref{eq:Thm4_S1} and \eqref{eq:Thm4_infty}, we get 
\[
\||T_{1}||\na T_{2}|(\tau)\|_{L_{x}^{2}}\lesssim A_{1;s_{1}}^{s_{1}/2}B_{2;\infty,s_{1}}^{s_{1}/2}A_{1;\infty}^{1-s_{1}/2}B_{2;\infty}^{1-s_{1}/2}e^{-a(1-s_{1}/2)v_{*}\tau}.
\]
Suppose
\begin{equation}
a(1-s_{1}/2)v_{*}\ge\la,
\end{equation}
 then we get 
\[
\||T_{1}||\na T_{2}|\|_{N(t)}\le A_{1;s_{1}}^{s_{1}/2}B_{2;\infty,s_{1}}^{s_{1}/2}A_{1;\infty}^{1-s_{1}/2}B_{2;\infty}^{1-s_{1}/2}\la^{-1}e^{-\la t}.
\]

\emph{Estimate (4).} The strategy is the same. For $s>0$ and $\ta\in[0,1]$,
\[
\||T_{2}||\na T_{1}|\|_{L_{x}^{s}}\le\|T_{2}\|_{L_{x_{1}}^{s/(1-\ta)}L_{x_{2}}^{s}}\|\na T_{1}\|_{L_{x_{1}}^{s/\ta}L_{x_{2}}^{\infty}}\lesssim A_{2;s/(1-\ta),s}B_{1;s/\ta}.
\]
For $(s/(1-\ta),s)\in\ml C_{A}^{(1,1)}$ and $s/\ta\in\ml C_{B}^{(1)}$,
we need
\[
s>\max\left(\fc{(2-\ta)\apa_{1}}2,\fc{\ta\apa_{1}}{2-\apa_{1}}\right).
\]
The ``max'' is minimized by letting $\ta=\fc{2(2-\apa_{1})}{4-\apa_{1}}$,
which gives $s>\fc{2\apa_{1}}{4-\apa_{1}}$, where the lower bound
is less than $2$. Hence 
\[
\||T_{2}||\na T_{1}|\|_{L_{x}^{s_{2}}}\lesssim A_{2;s_{2}(4-\apa_{1})/\apa_{1},s_{2}}B_{1;s_{2}(4-\apa_{1})/(4-2\apa_{1})}
\]
for some $s_{2}<2$. Next,
\begin{align*}
\||T_{2}||\na T_{1}|(\tau)\|_{L_{x}^{\infty}} & \le\sum_{k,j}\||R_{2;j}||\na R_{1;k}|(\tau)\|_{L_{x}^{\infty}}\\
 & \lesssim\sum_{k,j}\oa_{j}^{\fc 1{\apa_{1}}}\sa_{k}^{\fc 1{\apa_{1}}}\langle u_{k}\rangle e^{-av_{*}\tau}\\
 & \lesssim A_{2;\infty}B_{1;\infty}e^{-av_{*}\tau}.
\end{align*}
By interpolation we get
\[
\||T_{2}||\na T_{1}|(\tau)\|_{L_{x}^{2}}\lesssim A_{2;s_{2}(4-\apa_{1})/\apa_{1},s_{2}}^{s_{2}/2}B_{1;s_{2}(4-\apa_{1})/(4-2\apa_{1})}^{s_{2}/2}A_{2;\infty}^{1-s_{2}/2}B_{1;\infty}^{1-s_{2}/2}e^{-a(1-s_{2}/2)v_{*}\tau}.
\]
Suppose 
\begin{equation}
a(1-s_{2}/2)v_{*}\ge\la,
\end{equation}
then we get 
\[
\||T_{2}||\na T_{2}|\|_{N(t)}\lesssim A_{2;s_{2}(4-\apa_{1})/\apa_{1},s_{2}}^{s_{2}/2}B_{1;s_{2}(4-\apa_{1})/(4-2\apa_{1})}^{s_{2}/2}A_{2;\infty}^{1-s_{2}/2}B_{1;\infty}^{1-s_{2}/2}\la^{-1}e^{-\la t}.
\]

\textbf{Part 3. Estimate of $\|\na H_{2}\|_{N(t)}$.} Choose $2>s_{3}>\fc{2\apa_{1}}{\apa_{1}+2}$.
By Lemma \ref{lem:imp} (H1), we get 
\[
\|\na H_{2}(\tau)\|_{L_{x}^{2}}\lesssim(\sum_{i=1,2}A_{2;\apa_{i}q}^{\apa_{i}}B_{2;p})^{s_{3}/2}(\sum_{i=1,2}A_{2;\infty}^{\apa_{i}}B_{2;\infty})^{1-s_{3}/2}e^{-a\min(\apa_{1},1)(1-s_{3}/2)v_{*}\tau},
\]
where $p,q$ are arbitrary numbers in $(0,\infty]$ satisfying $\fc 1q+\fc 1p=\fc 1{s_{3}}$.
Since 
\[
\fc 1{s_{3}}<\fc 1{\apa_{1}}+\fc 12=1+(\fc 1{\apa_{1}}-\fc 12),
\]
we can choose $p,q$ such that $\fc 1q<1$ and $\fc 1p<\fc 1{\apa_{1}}-\fc 12$.
Thus $\apa_{1}q\in\ml C_{A}^{(2)}$ and $p\in\ml C_{B}^{(2)}$. Suppose
\begin{equation}
a\min(\apa_{1},1)(1-s_{3}/2)v_{*}\ge\la,
\end{equation}
then we get
\[
\|\na H_{2}\|_{N(t)}\lesssim(\sum_{i=1,2}A_{2;\apa_{i}q}^{\apa_{i}}B_{2;p})^{s_{3}/2}(\sum_{i=1,2}A_{2;\infty}^{\apa_{i}}B_{2;\infty})^{1-s_{3}/2}\la^{-1}e^{-\la t}.
\]

Combining all three parts, we get
\[
\|\na\Phi\eta\|_{S(t)}\lesssim \|\na G\|_{N(t)}+\|\na H_{1}\|_{N(t)}
+\|\na H_{2}\|_{N(t)}\lesssim \la^{-\mu}e^{-\lambda t} 
\]
for some $\mu>0$.
\end{proof}

\subsection{Construction of 1D-2D-3D trains \label{sub:PLP}}

In this subsection,
as our last main result, we construct 1D-2D-3D trains of the form
\begin{equation}
u=
T_{1}+\eta_{1}+T_{2}+\eta_{2}+T_{3}+\eta,\label{eq:3Din}
\end{equation}
where $T_{1}=T_{1}(t,x_{1})$, $T_{2}=T_{2}(t,x_{1},x_{2})$, and
$T_{3}=T_{3}(t,x)$ ($x=(x_{1},x_{2},x_{3})$) are 1D, 2D, and 3D
soliton train profiles respectively, with 
initial positions of all the solitons being the origin.
$\eta_{1}=\eta_{1}(t,x_{1})$
and $\eta_{2}=\eta_{2}(t,x_{1},x_{2})$ is such that $T_{1}+\eta_{1}+T_{2}+\eta_{2}$
is an 1D-2D mixed train (the fact that $T_{1}+\eta_{1}$ is itself
an 1D train will not be explicitly needed later), and $\eta=\eta(t,x)$
is the remaining error to be found. To be precise, let $T_{1}=\sum_{k}R_{1;k}$,
where $R_{1;k}$ have frequencies $\sa_{1;k}$ and velocities $(u_{1;k},0,0)$;
$T_{2}=\sum_{k}R_{2;k}$, where $R_{2;k}$ have frequencies $\sa_{2;k}$
and velocities $(u_{2;k,1},u_{2;k,2},0)$; and $T_{3}=\sum_{j}R_{3;j}$,
where $R_{3;j}$ has frequencies $\oa_{j}$ and velocities $(v_{j,1},v_{j,2},v_{j,3})$.
And we define 
\EQ{
v_{*}=\min(v_{*}(T_{1}),v_{*}(T_{2}),v_{*}(T_{3}),v_{*}(T_{1},T_{2}),v_{*}(T_{1},T_{3}),v_{*}(T_{2},T_{3})),
}
where the numbers in the min are defined by \eqref{eq:vstar} and
\eqref{eq:vstar_inter}.

\eqref{eq:3Din} can be visualized as a plane-line-point soliton train
in 3D space. It turns out to be the only mixed trains involving more
than two dimensions that we can construct. To see this, we first give
a discussion on the control of lower dimensional errors. 

As we stressed, supremum controls in $x$ for lower dimensional objects
are necessary in constructing mixed trains. For the previous theorems
on $e$D-$d$D trains, we use controls of the form 
\begin{equation}
\|\eta_{e}(t)\|_{L_{x}^{\infty}}\le e^{-\la t}\label{eq:purex}
\end{equation}
established in Section \ref{sec:single}. In fact, it is also possible
to use space-time controls of the form
\begin{equation}
\|\eta_{e}\|_{L_{t}^{p}L_{x}^{\infty}(t)}\le e^{-\la t},\label{eq:txctrl}
\end{equation}
for suitable $p$. In 1D space, since $(4,\infty)\in\ml A^{(1)}$,
we can obtain $\|\eta_{1}\|_{L_{t}^{4}L_{x}^{\infty}(t)}$ control
by constructing $T_{1}+\eta_{1}$ such that $\|\eta_{1}\|_{S(t)}$
has exponential decay in $t$. For $e=2,3$, since $r_{\max}^{(e)}>e$
(recall \eqref{eq:rMqm}), \eqref{eq:txctrl} can be obtained from
the exponential decay of $\|\na\eta_{e}\|_{S(t)}$ and some $\|\eta_{e}\|_{L_{t}^{q}L_{x}^{2}(t)}$
(e.g. \eqref{eq:norm2}) by Gagliardo-Nirenberg's inequality. For
$e\ge 4$, \eqref{eq:txctrl} is not available (unless controls
of even higher order derivatives of $\eta_{e}$ are considered,
which we did not pursue).

There is actually no definite reason we followed a route of using
\eqref{eq:purex} but not \eqref{eq:txctrl} in constructing $e$D-$d$D
trains. As to mixed trains involving more than two dimensions, all
the lower dimensional errors have to have spatial supremum controls.
As a consequence, thanks to Theorem \ref{thm:mixed1}, one sees that
\eqref{eq:3Din} becomes the only possible case, where we have type
\eqref{eq:purex} control of $\eta_{1}$ and type \eqref{eq:txctrl}
control of $\eta_{2}$. The details will be given in the proof of
Theorem \ref{thm:1D2D3D}. 

Since $T_{1}+\eta_{1}+T_{2}+\eta_{2}$ is assumed to be a solution,
the source term of $\Phi$ with respect to \eqref{eq:3Din} becomes
\[
f(T_{1}+\eta_{1}+T_{2}+\eta_{2}+T_{3}+\eta)-f(T_{1}+\eta_{1}+T_{2}+\eta_{2})-\sum_{j}f(R_{3;j}).
\]
We will write it as $G+H_{1}+H_{2}$, where
\begin{align*}
G & =f(T_{1}+\eta_{1}+T_{2}+\eta_{2}+T_{3}+\eta)-f(T_{1}+\eta_{1}+T_{2}+\eta_{2}+T_{3}),\\
H_{1} & =f(T_{1}+\eta_{1}+T_{2}+\eta_{2}+T_{3})-f(T_{1}+\eta_{1}+T_{2}+\eta_{2})-f(T_{3}),\\
H_{2} & =f(T_{3})-\sum_{j}f(R_{3;j}).
\end{align*}
Our main theorem is the following. 
\begin{thm}
\label{thm:1D2D3D} Let $d=3$, and $f$ satisfy Assumptions
\ref{assuF}, \hyperref[assuT]{(T)$_1$}, 
\hyperref[assuT]{(T)$_2$}, and \hyperref[assuT]{(T)$_3$}.
Suppose $1\le\apa_{1}< 4/3$ and 
$\apa_1\le\apa_{2}\le 4/3$. For any finite $\rho,t_0>0$, 
there is a constant $\la_{0}>0$ such that the following
holds: For $\la_{0}\le\la<\infty$, there exist solutions of \eqref{eq:NLSE} of the form \eqref{eq:3Din},
such that 
\begin{align}
\sup_{t\ge t_0}e^{\la t}\left\{ \|\eta_{1}(t)\|_{H_{x_{1}}^{1}\cap W_{x_{1}}^{1,\infty}}+\|\eta_2\|_{S(t)}+\|\na\eta_2\|_{S(t)}+\|\eta\|_{S(t)}\right\} \le\rho.
\end{align}
\end{thm}

\begin{rem*}
We can take $t_0=0$ if $\alpha_2 <4/3$, 
or if $\rho$ is sufficiently small. The highest dimension 
cannot be larger than $3$ in order to estimate terms of the
form $\||\eta_{1}+\eta_2|^{\ba}|T_{3}|^{\ga}\|_{N(t)}$
for $\ga=1$ and $\apa_1$ ($\ba>0$ is irrelevant). 
\end{rem*}

\begin{proof}
For $\la$ no less than some positive number, 
Theorem \ref{thm:mixed1} implies the existence of an 1D-2D 
train $T_{1}+\eta_{1}+T_{2}+\eta_{2}$ such that
\begin{equation}
\|\eta_{1}(t)\|_{H_{x_{1}}^{1}\cap W_{x_{1}}^{1,\infty}}
\le e^{-\la t},\label{eq:ctrlerr1}
\end{equation}
and 
\begin{equation}
\|\eta_{2}\|_{S(t)}+\|\na\eta_{2}\|_{S(t)}\le e^{-\la t}.\label{eq:ctrlerr2}
\end{equation}
We'll not exploit gradient estimates in this proof, 
and hence we don't need the control of $\na\eta_{1}$. 
The control for $\na\eta_{2}$ is needed merely to induce a type \eqref{eq:txctrl}
control of $\eta_{2}$, as we show now. Denote $x'=(x_{1},x_{2})$.
From \eqref{eq:ctrlerr2}, we have $\|\eta_{2}\|_{L_{t}^{\infty}L_{x'}^{2}(t)}\le e^{-\la t}$
and $\|\na\eta_{2}\|_{L_{t}^{4}L_{x'}^{4}}\le e^{-\la t}$. By the
Gagliardo-Nirenberg's inequality (cf. \eqref{eq:lbest}), for $0<p\le\infty$,
\begin{align*}
\|\eta_{2}\|_{L_{t}^{p}L_{x'}^{\infty}} & \lesssim\left\Vert \|\eta_{2}\|_{L_{x'}^{2}}^{1/3}\|\na\eta_{2}\|_{L_{x'}^{4}}^{2/3}\right\Vert _{L_{t}^{p}}\\
 & \le\left\Vert \|\eta_{2}\|_{L_{x'}^{2}}^{1/3}\right\Vert _{L_{t}^{\infty}}\left\Vert \|\na\eta_{2}\|_{L_{x'}^{4}}^{2/3}\right\Vert _{L_{t}^{p}}\\
 & =\|\eta_{2}\|_{L_{t}^{\infty}L_{x'}^{2}}^{1/3}\|\na\eta_{2}\|_{L_{t}^{2p/3}L_{x'}^{4}}^{2/3}.
\end{align*}
Letting $p=6$, we get 
\begin{equation}
\|\eta_{2}\|_{L_{t}^{6}L_{x'}^{\infty}(t)}\lesssim e^{-\la t}.\label{eq:t6if}
\end{equation}

Suppose $\eta\in S_{\la,t_{0}}$, $\|\eta\|_{S_{\la,t_{0}}}\le1$. We will
derive the suitable estimates of $\|G\|_{N(t)}$, $\|H_{1}\|_{N(t)}$
and $\|H_{2}\|_{N(t)}$ for $\|\Phi\eta\|_{S_{\la,t_{0}}}\le1$ in
the following.

\medskip
\textbf{Part 1. Estimate of $\|G\|_{N(t)}$.} By \eqref{eq:fineq0},
\begin{align*}
|G| 
  \lesssim\sum_{i=1,2}\left\{ |\eta||T_{1}+\eta_{1}+T_{2}+T_{3}|^{\apa_{i}}+|\eta||\eta_{2}|^{\apa_{i}}+|\eta|^{\apa_{i}+1}\right\} .
\end{align*}
For the first term, by \eqref{eq:ctrlerr1},
\begin{align*}
\||\eta||T_{1}+\eta_{1}+T_{2}+T_{3}|^{\apa_{i}}\|_{L_{t}^{1}L_{x}^{2}(t)} & \le\|\eta\|_{L_{t}^{1}L_{x}^{2}(t)}\|T_{1}+\eta_{1}+T_{2}+T_{3}\|_{L_{t}^{\infty}L_{x}^{\infty}}^{\apa_{i}}\\
 & \lesssim(A_{1;\infty}+e^{-\la t_{0}}+A_{2;\infty}+A_{3;\infty})^{\apa_{i}}\la^{-1}e^{-\la t}.
\end{align*}
For the second, by \eqref{eq:t6if}, we have 
\begin{align*}
\||\eta||\eta_{2}|^{\apa_{i}}\|_{L_{t}^{1}L_{x}^{2}(t)} & \le\|\eta\|_{L_{t}^{6/(6-\apa_{i})}L_{x}^{2}(t)}\||\eta_{2}|^{\apa_{i}}\|_{L_{t}^{6/\apa_{i}}L_{x}^{\infty}(t)}\\
 & \lesssim(\la^{-(6-\apa_{i})/6}e^{-\apa_{i}\la t_{0}})e^{-\la t}.
\end{align*}
For the third, since $\apa_{2}\le4/3$, Lemma \ref{lem:main} (N0)
implies
\[
\||\eta|^{\apa_{i}+1}\|_{N(t)}\lesssim(\la^{-1+3\apa_{i}/4}e^{-\apa_{i}\la t_{0}})e^{-\la t}.
\]
(If $\apa_{2}=4/3$, for $i=2$ the smallness of the coefficient 
by letting $\la$ large relies on the assumption $t_0>0$.) 

\medskip

\textbf{Part 2. Estimate of $\|H_{1}\|_{N(t)}$.} Let $W=T_{1}+\eta_{1}+T_{2}+\eta_{2}$.
By Corollary \ref{cor:decomp},
\[
|H_{1}|=|f(W+T_{3})-f(W)-f(T_{3})|\lesssim\sum_{i=1,2}(|W|^{\apa_{i}}|T_{3}|+|W||T_{3}|^{\apa_{i}}).
\]
Thus we have to estimate (1) $|\eta_{1}|^{\ba}|T_{3}|^{\ga}$, (2)
$|\eta_{2}|^{\ba}|T_{3}|^{\ga}$, (3) $|T_{1}|^{\ba}|T_{3}|^{\ga}$,
and (4) $|T_{2}|^{\ba}|T_{3}|^{\ga}$, for $(\ba,\ga)=(\apa_{i},1)$
and $(1,\apa_{i})$. Notice that by assumption both 
$\ba,\ga\ge1$ in any case.

\emph{Estimate (1). }
\[
\||\eta_{1}|^{\ba}|T_{3}|^{\ga}\|_{L_{t}^{1}L_{x}^{2}(t)}\le\||\eta_{1}|^{\ba}\|_{L_{t}^{1}L_{x}^{\infty}(t)}\||T_{3}|^{\ga}\|_{L_{t}^{\infty}L_{x}^{2}(t)}\lesssim A_{3;2\ga}^{\ga}\|\eta_{1}\|_{L_{t}^{\ba}L_{x}^{\infty}(t)}^{\ba},
\]
where $2\ga\in\ml C_{A}^{(3)}$ since $\ga\ge 1$ and $\apa_{1}<4/3$.
(Notice that, for $\ga=1$ and $\apa_1$,
$2\ga\notin\ml{C}_A^{(d)}$ if $d\ge 4$. This is why we can only
consider $3$ as the highest dimension.)
By \eqref{eq:ctrlerr1},
\[
\|\eta_{1}\|_{L_{t}^{\ba}L_{x}^{\infty}(t)}^{\ba}\le\int_{t}^{\infty}e^{-\ba\la\tau}\, d\tau=(\ba\la)^{-1}e^{-\ba\la t}.
\]
Hence 
\[
\||\eta_{1}|^{\ba}|T_{3}|^{\ga}\|_{L_{t}^{1}L_{x}^{2}(t)}\lesssim(A_{3;2\ga}^{\ga}\la^{-1}e^{-(\ba-1)\la t_{0}})e^{-\la t}.
\]

\emph{Estimate (2).} As above we get 
\[
\||\eta_{2}|^{\ba}|T_{3}|^{\ga}\|_{L_{t}^{1}L_{x}^{2}(t)}\lesssim A_{3;2\ga}^{\ga}\|\eta_{2}\|_{L_{t}^{\ba}L_{x}^{\infty}}^{\ba}.
\]
Let $u(\tau)=\|\eta_{2}(\tau)\|_{L_{x}^{\infty}}$. \eqref{eq:t6if}
says $\|u\|_{L^{6}([t,\infty))}\lesssim e^{-\la t}$. Since $\ba<6$,
Proposition \ref{prop:subadest} implies 
\[
\|\eta_{2}\|_{L_{t}^{\ba}L_{x}^{\infty}(t)}=\|u\|_{L^{\ba}([t,\infty))}\lesssim\la^{\fc 16-\fc 1{\ba}}e^{-\la t}.
\]
Thus 
\[
\||\eta_{2}|^{\ba}|T_{3}|^{\ga}\|_{L_{t}^{1}L_{x}^{2}(t)}\lesssim(A_{3;2\ga}^{\ga}\la^{-1+\ba/6}e^{-(\ba-1)\la t_{0}})e^{-\la t}.
\]

\emph{Estimate (3).} Since
\[
|T_{1}|^{\apa_{i}}|T_{3}|+|T_{1}||T_{3}|^{\apa_{i}}\lesssim|T_{1}||T_{3}|(|T_{1}|+|T_{3}|)^{\apa_{i}-1}
\]
and 
\[
\|(|T_{1}|+|T_{3}|)^{\apa_{i}-1}\|_{L_{t}^{\infty}L_{x}^{\infty}(t)}\lesssim(A_{1;\infty}+A_{3;\infty})^{\apa_{i}-1},
\]
it suffices to estimate $\||T_{1}||T_{3}|\|_{N(t)}$.

\emph{Step 1.} For $s\in(0,\infty]$ and $\ta\in[0,1]$, with $x''=(x_2,x_3)$,
\begin{equation}
\||T_{1}||T_{3}|\|_{L_{x}^{s}}\le\|T_{1}\|_{L_{x_{1}}^{s/\ta}L_{x''}^{\infty}}\|T_{3}\|_{L_{x_{1}}^{s/(1-\ta)}L_{x''}^{s}}\lesssim A_{1;s/\ta}A_{3;s/(1-\ta),s}.\label{eq:Thm5_T1T3s}
\end{equation}
Here $A_{3;s/(1-\ta),s}$ is with respect to $(e,d)=(1,3)$.
We need $s/\ta\in\ml C_{A}^{(1)}$ and 
$(s/(1-\ta),s)\in\ml C_{A}^{(1,2)}$, that is
\[
s>\max\left(\fc{\ta\apa_{1}}2,\fc{(3-\ta)\apa_{1}}2\right).
\]
The ``max'' is minimized by letting $\ta=1$, which gives 
$s>\apa_{1}$. Since $\apa_{1}<2$, \eqref{eq:Thm5_T1T3s} gives
\begin{equation}
\||T_{1}||T_{3}|\|_{L_{x}^{s_1}}\lesssim 
A_{1;s_{1}}A_{3;\infty,s_{1}}\label{eq:Thm5_T1T3s1}
\end{equation}
for some $s_{1}<2$, with $s_1\in \ml{C}_{A}^{(1)}$ and 
$(\infty,s_{1})\in \ml{C}_{A}^{(1,2)}$.

\textit{Step 2.} Following the derivation of
\eqref{eq:T3_H1_2-4-1}, we get 
\begin{equation}
\||T_{1}||T_{3}|(\tau)\|_{L_{x}^{\infty}}\le\sum_{k,j}|R_{1;k}||R_{3;j}|(\tau)\lesssim A_{1;\infty}A_{3;\infty}e^{-av_{*}\tau}.\label{eq:Thm5_T1T3i}
\end{equation}

From \eqref{eq:Thm5_T1T3s1} and \eqref{eq:Thm5_T1T3i}, we get 
\[
\||T_{1}||T_{3}|(\tau)\|_{L_{x}^{2}}\lesssim A_{1;s_{1}}^{s_{1}/2}A_{3;\infty,s_{1}}^{s_{1}/2}A_{1;\infty}^{1-s_{1}/2}A_{3;\infty}^{1-s_{1}/2}e^{-a(1-s_{1}/2)v_{*}\tau}.
\]
Suppose 
\EQ{
a(1-s_{1}/2)v_{*}\ge\la,
}
we get 
\[
\||T_{1}||T_{3}|\|_{N(t)}\lesssim A_{1;s_{1}}^{s_{1}/2}A_{3;\infty,s_{1}}^{s_{1}/2}A_{1;\infty}^{1-s_{1}/2}A_{3;\infty}^{1-s_{1}/2}\la^{-1}e^{-\la t}.
\]

\emph{Estimate (4).} As above, it suffices to estimate $\||T_{2}||T_{3}|\|_{N(t)}$.
Let $x'=(x_{1},x_{2})$. For $s\in(0,\infty]$ and $\ta\in[0,1]$,
\begin{equation}
\||T_{2}||T_{3}|\|_{L_{x}^{s}}\le\|T_{2}\|_{L_{x'}^{s/\ta}L_{x_{3}}^{\infty}}\|T_{3}\|_{L_{x'}^{s/(1-\ta)}L_{x_{3}}^{s}}\lesssim A_{2;s/\ta}A_{3;s/(1-\ta),s}.\label{eq:Thm5_T2T3s}
\end{equation}
Here $A_{3;s/(1-\ta),s}$ is with respect to $(e,d)=(2,3)$.
For $s/\ta\in\ml C_{A}^{(2)}$ and $(s/(1-\ta),s)\in\ml C_{A}^{(2,1)}$,
we need 
\[
s>\max\left(\apa_{1}\ta,(\fc 32-\ta)\apa_{1}\right).
\]
The ``max'' is minimized by letting $\ta=3/4$, which gives $s>\fc{3\apa_{1}}4$.
The lower bound is less than $2$. Thus there is $s_{2}<2$ such that
\[
\||T_{2}||T_{3}|\|_{L_{x}^{s_{2}}}\lesssim A_{2;4s_{2}/3}A_{3;4s_{2},s_{2}},
\]
with $4s_{2}/3\in\ml C_{A}^{(2)}$ and $(4s_{2},s_{2})\in\ml C_{A}^{(2,1)}$.

By \eqref{eq:T3_H1_2-4-1},
we have 
\[
\||T_{2}||T_{3}|(\tau)\|_{L_{x}^{\infty}}\lesssim A_{2;\infty}A_{3;\infty}e^{-av_{*}\tau}.
\]

By interpolation we get the $L_{x}^{2}$ estimate. And if $a(1-s_{2}/2)v_{*}\ge\la$,
we get 
\[
\||T_{2}||T_{3}|\|_{N(t)}\lesssim A_{2;4s_{2}/3}^{s_{2}/2}A_{3;4s_{2},s_{2}}^{s_{2}/2}A_{2;\infty}^{1-s_{2}/2}A_{3;\infty}^{1-s_{2}/2}\la^{-1}e^{-\la t}.
\]

\textbf{Part 3. Estimate of $\|H_{2}\|_{N(t)}$.} Choose $2>s_{3}>\fc{3\apa_{1}}{2(\apa_{1}+1)}$.
By Lemma \ref{lem:imp} (H0), 
\[
\|H_{2}(\tau)\|_{L_{x}^{2}}\lesssim
(\sum_{i=1,2}A_{3;(\apa_{i}+1)s_3}^{\apa_{i}+1})^{s_{3}/2}(\sum_{i=1,2}A_{3;\infty}^{\apa_{i}+1})^{1-s_{3}/2}e^{-a(1-s_{3}/2)v_{*}\tau},
\]
with $(\apa_{1}+1)s_{3}\in\ml C_{A}^{(3)}$. 
Suppose $a(1-s_{3}/2)v_{*}\ge\la$, we get 
\[
\|H_{2}\|_{L_{t}^{1}L_{x}^{2}(t)}\lesssim
(\sum_{i=1,2}A_{3;(\apa_{i}+1)s_3}^{\apa_{i}+1})^{s_{3}/2}(\sum_{i=1,2}A_{3;\infty}^{\apa_{i}+1})^{1-s_{3}/2}\la^{-1}e^{-\la t}.
\]

Combining all three parts, we see $\|\Phi\eta\|_{S(t)}\le e^{-\la t}$
for $\la$ large enough with suitable frequencies and velocities 
of the solitons.

\end{proof}

\begin{appendices}

\section{}\label{appI}

We prove Lemma \ref{lem:compete} in this appendix. 
We will consider slightly more general forms of
the assertions, so that they actually cover the 
anisotropic cases used in Section \ref{sec:MixedT}. 

For $p_{1},p_{2}\in(0,\infty)$, 
sequence $\{\oa_{j}\}_{j\in\mb N}$ in $(0,\oa_{*})$, 
and sequence $\{v_{j}\}_{j\in\mb N}$ in $\mb R^{d}$,
define 
\begin{align*}
\wt A_{p_{1},p_{2}} & =\wt A_{p_{1},p_{2}}(\{\oa_{j}\})\coloneqq(\sum_{j}\oa_{j}^{p_{1}p_{2}})^{1/p_{1}}\\
\wt B_{p_{1},p_{2}} & =\wt B_{p_{1},p_{2}}(\{\oa_{j}\},\{v_{j}\})\coloneqq(\sum_{j}\langle v_{j}\rangle^{p_{1}}\oa_{j}^{p_{1}p_{2}})^{1/p_{1}}.
\end{align*}

\begin{prop*}
Given $0<p_{1}\le q_{1}<\infty$ and $0<p_{2}<q_{2}<\infty$. We have
\begin{align*}
\wt A_{q_{1},q_{2}} & <\oa_{*}^{q_{2}-p_{2}}\wt A_{p_{1},p_{2}},\quad\mbox{if}\quad\wt A_{q_{1},q_{2}}<\infty,\\
\wt B_{q_{1},q_{2}} & <\oa_{*}^{q_{2}-p_{2}}\wt B_{p_{1},p_{2}},\quad\mbox{if}\quad\wt B_{q_{1},q_{2}}<\infty.
\end{align*}
\end{prop*}

\begin{rem*}
By letting $p_{1}=\min(1,p)$, $p_{2}=\fc 1{\apa_{1}}-\fc d{2p}$,
$q_{1}=\min(1,q)$, $q_{2}=\fc 1{\apa_{1}}-\fc d{2q}$, and notice
that 
\[
\oa_{*}^{q_{2}-p_{2}}\le\max(1,\oa_{*})^{q_{2}-p_{2}}\le\max(1,\oa_{*})^{q_{2}}\le\max(1,\oa_{*})^{1/\apa_{1}},
\]
we get Lemma \ref{lem:compete} (a).
\end{rem*}

\begin{proof}
We have
\begin{align*}
\wt A_{q_{1},q_{2}} 
&=\Bigl[
(\sum_{j}\oa_{j}^{q_{1}q_{2}})^{p_{1}/q_{1}}
\Bigr]^{1/p_{1}}
\le(\sum_{j}\oa_{j}^{p_{1}q_{2}})^{1/p_{1}}
\quad(\mbox{since }p_{1}/q_{1}\le1)\\
&=\oa_{*}^{q_{2}}(\sum_{j}(\oa_{j}/\oa_{*})^{p_{1}q_{2}})^{1/p_{1}}
<\oa_{*}^{q_{2}}(\sum_{j}(\oa_{j}/\oa_{*})^{p_{1}p_{2}})^{1/p_{1}}\quad(\mbox{since }\oa_{j}/\oa_{*}<1)\\
& =\oa_{*}^{q_{2}-p_{2}}\wt A_{p_{1},p_{2}}.
\end{align*}

Similarly, 
\begin{align*}
\wt B_{q_{1},q_{2}} 
 \le(\sum_{j}\langle v_{j}\rangle^{p_{1}}\oa_{j}^{p_{1}q_{2}})^{1/p_{1}}
  <\oa_{*}^{q_{2}}(\sum_{j}\langle v_{j}\rangle^{p_{1}}(\oa_{j}/\oa_{*})^{p_{1}p_{2}})^{1/p_{1}}
  =\oa_{*}^{q_{2}-p_{2}}\wt B_{q_{1},q_{2}}.
\end{align*}

\end{proof}
Let $v_{*}$ be as defined by \eqref{eq:vstar}. 
Lemma \ref{lem:compete} (b) is
a corollary of the following
\begin{prop*}
Given $0<p_{1},q_{1},q_{2}<\infty$ and $1/2<p_{2}<\infty$. For any
constants $c,\La>0$, there exist $\{\oa_{j}\}$ and $\{v_{j}\}$
such that $\wt A_{q_{1},q_{2}},\wt B_{p_{1},p_{2}}\le c$ and $v_{*}\ge\La$. \end{prop*}
\begin{proof}
For constants $0<\rho<1$, $\ga>0$, and $\da\ge0$, let $\oa_{j}=\oa_{*}\rho^{2j}$,
and $v_{j}$ satisfies 
\[
|v_{j}|=\ga(\sum_{\ell=2}^{j}\rho^{-\ell})+\da.
\]
(The empty summation $\sum_{\ell=2}^{1}$ is understood to be zero.)
Then for $j<k$ we have
\[
\min(\oa_{j}^{1/2},\oa_{k}^{1/2})|v_{j}-v_{k}|\ge\oa_{k}^{1/2}(|v_{k}|-|v_{j}|)=\oa_{*}^{1/2}\rho^{k}\cdot\ga(\sum_{\ell=j+1}^{k}\rho^{-\ell}).
\]
Since $\rho^{k}(\sum_{\ell=j+1}^{k}\rho^{-\ell})>1$ ($\forall\rho\in(0,1)$),
$v_{*}\ge\La$ as long as 
\[
\ga\ge2\oa_{*}^{-1/2}\La.
\]
To complete the proof, it suffices to show that $\wt A_{q_{1},q_{2}},\wt B_{p_{1},p_{2}}\to0$
as $\rho\to0$. For $\wt A_{q_{1},q_{2}}$, we have
\[
\lim_{\rho\to0}\wt A_{q_{1},q_{2}}^{q_{1}}=\oa_{*}^{q_{1}q_{2}}\lim_{\rho\to0}\sum_{j}\rho^{2q_{1}q_{2}j}=0.
\]
On the other hand, since $\langle v_{j}\rangle\lesssim|v_{j}|+1=\ga(\sum_{\ell=2}^{j}\rho^{-\ell})+(\da+1)$,
\begin{align*}
\wt B_{p_{1},p_{2}}^{p_{1}} & \lesssim\sum_{j}\Bigl[\ga^{p_{1}}(\sum_{\ell=2}^{j}\rho^{-\ell})^{p_{1}}+(\da+1)^{p_{1}}\Bigr]\oa_{*}^{p_{1}p_{2}}\rho^{2p_{1}p_{2}j}\\
 & =\ga^{p_{1}}\oa_{*}^{p_{1}p_{2}}I(\rho)+(\da+1)^{p_{1}}\wt A_{p_{1},p_{2}}^{p_{1}},
\end{align*}
where 
\begin{align*}
I(\rho) & =\sum_{j}(\sum_{\ell=2}^{j}\rho^{-\ell})^{p_{1}}\oa^{2p_{1}p_{2}j}
=\sum_{j}(\sum_{\ell=2}^{j}\rho^{-\ell})^{p_{1}}\rho^{p_{1}j}\cdot\rho^{-p_{1}j}\oa^{2p_{1}p_{2}j}\\
 & =\sum_{j}(\sum_{\ell=2}^{j}\rho^{j-\ell})^{p_{1}}\rho^{2p_{1}(p_{2}-1/2)j}
\le\sum_{j}(\sum_{\ell=0}^{\infty}\rho^{\ell})^{p_{1}}\rho^{2p_{1}(p_{2}-1/2)j}\\
 & =(1-\rho)^{-p_{1}}\cdot\fc{\rho^{2p_{1}(p_{2}-1/2)}}{1-\rho^{2p_{1}(p_{2}-1/2)}}\to 0\quad\mbox{as}\quad \rho \to 0.
\end{align*}
\end{proof}

\section{}\label{appII}

Let $x=(x',x'')$ be as in Section \ref{sub:eDdDtrain}. One would
wonder if the $L_{x'}^{p}L_{x''}^{q}$ norm can be bounded by the
$L_{x}^{p}\cap L_{x}^{q}$ norm. This is in general not the case.
Consider a function $u:\mb R^{2}\to\mb R$ of the form 
\[
u(x,y)=1_{0<x<1}|x|^{ma}\psi(|x|^{a}y),
\]
where $m,a$ are real parameters, $\psi\in C_{c}^{\infty}(\mb R)$.
Then for $p,q\in(0,\infty)$ we have 
\begin{align*}
\|u\|_{L_{xy}^{p}} & =\|\psi\|_{L^{p}}(\int_{0}^{1}|x|^{ap(m-1/p)}\, dx)^{1/p}\\
\|u\|_{L_{xy}^{q}} & =\|\psi\|_{L^{q}}(\int_{0}^{1}|x|^{aq(m-1/q)}\, dx)^{1/q}\\
\|u\|_{L_{x}^{p}L_{y}^{q}} & =\|\psi\|_{L^{q}}(\int_{0}^{1}|x|^{ap(m-1/q)}\, dx)^{1/p}.
\end{align*}
Suppose $p>q$. Then if $0<m<1/q$, there exists $a>0$ such that 
\[
ap(m-\fc 1q)<-1<\min(ap(m-\fc 1p),aq(m-\fc 1q)),
\]
which implies $u\in L_{xy}^{p}\cap L_{xy}^{q}$ but $\|u\|_{L_{x}^{p}L_{y}^{q}}=\infty$. 

\end{appendices}

\section*{Acknowledgements}
We thank Prof.~Tetsu Mizumachi for fruitful discussions in the initial stage of this project.
Tsai's research is supported in part by NSERC grant 261356-13.


\end{document}